\documentclass[10pt]{amsart}

\usepackage{amssymb,amsthm,amsfonts}
\usepackage{bm,dsfont}
\usepackage{enumerate}
\usepackage{mathrsfs}
\usepackage[linktocpage]{hyperref}
\usepackage{graphicx}
\graphicspath{ {images/} }
\usepackage{tikz}
\usepackage{xcolor}

\newtheorem{thm}{Theorem}[section]
\newtheorem{theorem}[thm]{Theorem}

\newtheorem{corollary}[thm]{Corollary}

\newtheorem{proposition}[thm]{Proposition}
\newtheorem{lemma}[thm]{Lemma}

\newtheorem*{remark}{Remark}

\theoremstyle{definition}
\newtheorem{definition}[thm]{Definition}

\hypersetup{
     colorlinks   = true,
     citecolor    = blue
}
\hypersetup{linkcolor = red}
\usepackage{cite}

\def\m{{\mathbf{m}}}

\def\R{{\mathbb{R}}}
\def\Z{{\mathbb{Z}}}

\newcommand{\QQ}{\mathsf{Q}}
\newcommand{\RR}{\mathcal{R}}
\newcommand{\B}{\mathscr{B}}
\newcommand{\M}{\mathfrak{M}}

\newcommand{\C}{\mathbb{C}}
\newcommand{\D}{\mathcal{D}}
\newcommand{\DD}{\mathscr{D}}

\newcommand{\T}{\mathbb{T}}

\newcommand{\LL}{\mathcal{L}}

\newcommand{\I}{\mathcal{I}}

\newcommand{\F}{\mathcal{F}}

\newcommand{\K}{\mathcal{K}}

\newcommand{\OO}{\mathcal{O}}
\newcommand{\PP}{\mathsf{P}}
\newcommand{\PPP}{\mathscr{P}}

\newcommand{\s}{\textbf{s}}

\newcommand{\cc}{\textbf{c}}
\newcommand{\pp}{\mathfrak{p}}

	\newcommand{\norm}[1]{\left\|#1\right\|}

				\address{\parbox{\linewidth}{Department of Mathematics, University of Maryland, College Park, MD 20742, USA\\
				\\
			\emph{Current address:} Faculty of Mathematics and Computer Science, Nicolaus Copernicus University, Chopina 12/18, 87-100 Toru\' n, Poland}}			
		\email{mkim16@mat.umk.pl}

\date{\today}

\title[Limit theorems for higher rank actions]{Limit theorems for higher rank actions on Heisenberg nilmanifolds}
\author[Minsung Kim]{Minsung Kim*}

\begin{document}
\begin{abstract} 
 The main result of this paper is a construction of finitely additive measures for higher rank abelian actions on Heisenberg nilmanifolds. Under a full measure set of Diophantine conditions for the generators of the action, we construct \emph{Bufetov functionals} on rectangles on $(2g+1)$-dimensional Heisenberg manifolds. We prove that deviation of the ergodic integral of higher rank actions is described by the asymptotic of Bufetov functionals for a sufficiently smooth function. As a corollary, the distribution of normalized ergodic integrals which have variance 1, converges along certain subsequences to a non-degenerate compactly supported measure on the real line.
\end{abstract}

\maketitle

\tableofcontents

\section{Introduction}
\subsection{Introduction}
The asymptotic behavior and limiting distribution of ergodic averages of translation flows were studied by A. Bufetov in his series of works \cite{bufetov2014finitely, bufetov2010holder, B14}. He constructed finitely-additive H\"older measures and cocycles over translation flows that are known as \emph{Bufetov functionals}. Such functionals are constructed to derive the deviation of ergodic integrals of translation flows by providing a new proof of the celebrated work of G. Forni \cite{forni2002deviation}. Bufetov observed a relation between his finitely-additive H\"older measures and Forni's invariant distributions which already appeared in the work of solving cohomological equations. Following these observations, a bijection between the space of such functionals and the space of invariant distributions was constructed analogously for other parabolic flows and higher rank actions (e.g. \cite{bufetov2013, BS13, BF14, forni2020time,forni2020equidistribution}). 
It also turned out that his construction of such functionals is closely related to limit shapes of ergodic sums for interval exchange transformations \cite{MMY10,SUY20}.

In this paper, our main results are on effective equidistributions for higher rank abelian actions and limit theorems on Heisenberg nilmanifolds. We firstly introduce the construction of the Bufetov functional for higher rank abelian actions.  The main argument is based on the renormalization method and induction argument  inspired by the work of Cosentino and Flaminio \cite{CF15}. 
Then we prove that the functional exists for a full measure set of frame $\alpha$ under a sufficient Diophantine condition for the recurrence (see \S \ref{sec;dio}).
Likewise, we  construct the bijections between functionals and invariant currents,  which appeared in \cite{CF15} to obtain a new deviation formula.
As a corollary, we prove the existence of the limit distributions of (normalized) ergodic integrals for higher rank actions. More specifically, as a random variable, the normalized ergodic integrals converge in distribution along certain subsequences to a non-degenerate, compactly supported measure on the real line. 

It is natural to consider possible applications to the limit theorems of theta sums. It is known that the structure of  Heisenberg flows (or quasi-abelian flows on higher steps, respectively), as a first return map on a transverse torus can be expressed as a skew-translation map. These relations provide the correspondence between the bound of ergodic integrals of some nilflows and  exponential sums (so called \emph{Weyl sums}). In several contexts, such bound of Weyl sums and effective equidistribution of nilflows were discovered and compared independently \cite{FJK77,FF06, T15, F16}. Furthermore, their limit theorems were studied in a similar context by several authors as well \cite{griffin2014limit, cellarosi2016quadratic, forni2020time, CGO22}. 
As a minor application on multi-parameters, we obtain limit theorems for theta series on Siegel half spaces, which generalize previous results \cite{GG04,M99} (see  \cite{tolimieri1978heisenberg,mumford2007tata1, mumford2007tata} for a general introduction). 

In the later sections, an $L^2$-lower bound of the Bufetov functional on a transverse torus and its analyticity on the higher dimensional rectangular domain are obtained. These parts indicate that our functionals do not vanish on a certain transverse torus and
they extend to the complex domain in the weighted space (of fast-decaying coefficient of functionals). From these results, polynomial type lower bound  for the analytic function are obtained on certain sub-level sets.
 This tool was originally devised by Forni and Kanigowski  to study the bound of correlations for time-changes of Heisenberg nilflows \cite{forni2020time}. Unfortunately, it turned out that we can not obtain similar results for higher rank abelian actions because the time-changes of higher rank actions on Heisenberg nilmanifolds are all \emph{trivial}. This follows from the triviality of the first cohomology group (see \cite[Theorem 3.16]{CF15}). Therefore the time-changes of higher rank actions are conjugated to the linear action and never mixing.

Regarding mixing properties for time-changes of nilflows, they were firstly studied for Heisenberg nilflows \cite{AFU11}. Then, the method was extended to non-trivial time-changes for any uniquely ergodic nilflows on general nilmanifolds \cite{Rav18, avila2021mixing}. On time-changes of Heisenberg flows, multiple mixing and disjointness \cite{forni2020multiple} were also obtained by verifying a certain type of divergence of orbits, so called Ratner's property.  In $\Z^k$-actions, mixing of shapes for automorphisms on nilmanifolds was proved \cite{gorodnik2014exponential, gorodnik2015mixing}.

 Comprehensive studies of temporal limit theorems for horocycle flows were carried out in the work of D. Dolgopyat and O. Sarig \cite{DS17}. Recent work of Ravotti \cite{Rav21} reproved spatial  and temporal limit theorems for horocycle flows  independently, by following Ratner's argument, not relying on the methods for invariant distributions (or currents). 
It is still unknown if temporal (or weaker) limit theorems for nilflows can be proved, and it will be a possible subject of further works.

\subsection{Definitions and statement of results} We review the definitions about Heisenberg manifold and its moduli space.

\subsubsection{Heisenberg manifold}
Let $g\geq 1$ and $\mathsf{H}^g$ be the standard $2g+1$ dimensional Heisenberg group and set $\mathsf{\Gamma}:= \Z^g \times \Z^g \times \frac{1}{2} \Z$ a discrete and co-compact subgroup of $\mathsf{H}^g$. We shall call it standard lattice of $\mathsf{H}^g$ and the quotient $M := \mathsf{H}^g /\mathsf{\Gamma} $ will be called \emph{Heisenberg manifold}. Lie algebra $\mathfrak{h}^g = Lie(\mathsf{H}^g)$ is equipped with a basis $ (X_1, \cdots, X_g, Y_1, \cdots, Y_g, Z)$ satisfying canonical commutation relations
\begin{equation}\label{commute}
[X_i,Y_j] = \delta_{ij}Z, \quad \text{ where } \delta_{ij} = 1 \text{ if } i=j, \text{ and } \delta_{ij} = 0 \text{ otherwise.}
\end{equation}

For $1 \leq d \leq g$, let $\mathsf{P} := \mathsf{P}^d$ be a subgroup of $\mathsf{H}^g$ where its Lie algebra $\mathfrak{p}:= Lie(\mathsf{P}^d)$ is generated by $\{X_1, \cdots, X_d\}$. For any $\alpha \in Sp_{2g}(\R)$, set $(X_i^\alpha, Y_i^\alpha, Z) = \alpha^{-1}(X_i,Y_i,Z)$ for $1 \leq i \leq d$. We define a parametrization of the subgroup $\PP^{d,\alpha} = \alpha^{-1}(\PP^d)$ according to
$$\PP_x^{d,\alpha} := \exp(x_1X_1^\alpha+\cdots+x_dX_d^\alpha), \ \ x = (x_1, \cdots, x_d) \in \R^d.$$

By central extension of $\R^{2g}$ by $\R$, we have an exact sequence
$$0 \rightarrow Z(\mathsf{H}^g) \rightarrow \mathsf{H}^g \rightarrow \R^{2g} \rightarrow 0.$$
The natural projection  $pr: M \rightarrow \mathsf{H}^g /(\mathsf{\Gamma}Z(\mathsf{H}^g)) $ maps $M$ onto a $2g$-dimensional torus $\T^{2g}:= \R^{2g}/\Z^{2g}$.

\subsubsection{Moduli space}
The group of automorphisms of $\mathsf{H}^g$ that are trivial on the center is denoted by $Aut_0(\mathsf{H}^g) = Sp_{2g}(\R) \ltimes \R^{2g}$. Since dynamical properties of actions are invariant under inner automorphism, we restrict our interest to $Sp_{2g}(\R)$. We regard $Sp_{2g}(\R)$ as the \emph{deformation space} of the standard Heisenberg manifold $M$ and  call the quotient $\M_g = Sp_{2g}(\R) / Sp_{2g}(\Z)$ the \emph{moduli space} of (standard) Heisenberg manifold.

\emph{Siegel modular variety} is a double coset space $\Sigma_g = K_g\backslash Sp_{2g}(\R)/ Sp_{2g}(\Z)$ where $K_g$ is a maximal compact subgroup $Sp_{2g}(\R) \cap SO_{2g}(\R)$ of $Sp_{2g}(\R)$.
For $\alpha \in Sp_{2g}(\R) $, we denote $[\alpha]:=\alpha Sp_{2g}(\Z) $ by its projection on the moduli space $\M_g$ and write $[[\alpha]]:= K_g\alpha Sp_{2g}(\Z)$ the projection of $\alpha$ to the Siegel modular variety $\Sigma_g$.

Double coset $K_g\backslash Sp_{2g}(\R)/ 1_{2g}$ is identifed to the Siegel upper half space $\mathfrak{H}_g:= \{Z \in Sym_g(\C) \mid \Im(Z) >0 \}$. \emph{Siegel upper half space} of genus $g$ is a complex manifold of symmetric complex $g\times g$ matrices $Z = X+iY$ with positive-definite symmetric imaginary part $\Im(Z) = Y$ and arbitrary (symmetric) real part $X$. Denote by $\Sigma_g \thickapprox Sp_{2g}(\Z) \backslash \mathfrak{H}_g$.

\subsubsection{Representation}\label{sec;repr}
We write the Hilbert sum decomposition of
\begin{equation}\label{rep}
L^2(M) =  \bigoplus_{n\in\Z}H_n
\end{equation}
into closed $\mathsf{H}^g$-invariant subspaces. Set $f = \sum_{n \in \Z}f_n \in L^2(M)$ and $f_n \in H_n$ where
$$H_n = \{f \in L^2(M) \mid \exp(tZ)f = \exp(2\pi \iota n K t)f\}$$
for some fixed $K>0$. The center $Z(\mathsf{H}^g)$ has spectrum $2\pi \Z \backslash \{0\}$ on $L^2_0(M)$, the space of zero mean. Then the space $L^2(M)$ splits as Hilbert sum of $\mathsf{H}^g$-module $H_n$, which is equivalent to irreducible representation $\pi$.

By Stone-Von Neumann theorem, the unitary irreducible representation of the Heisenberg group of non-zero central parameter $K>0$ is unitarily equivalent to the Schr\"odinger representation  $\pi$.
By differentiating the Schr\"odinger representation, we obtain a representation of the Lie algebra $\mathfrak{h}^g$ on Schwartz space $\mathscr{S}(\R^g) \subset L^2(\R^g)$ (as a $C^\infty$-vector).  This is called \emph{infinitesimal derived representation} $d\pi_*$ with parameter $n \in \Z$, and for each $k = 1, 2, \cdots, g$,
\begin{equation}\label{def;derived}
d\pi_*(X_k) = \frac{\partial{}}{\partial{x_k}}, \quad d\pi_*(Y_k) = 2\pi \iota n K x_k, \quad d\pi_*(Z) = 2\pi \iota n K
\end{equation}
acts on $L^2(\R^g) \simeq L^2(H_n)$.

Given a basis $(V_i)$ of the Lie algebra, we set a Laplacian $\Delta = -\sum_iV_i^2$ and define $L^2$-Sobolev norm $\norm{f}_s^2 =  \langle f, (1+\Delta)^sf \rangle$ where $\langle \cdot, \cdot \rangle$ is an ordinary inner product.
For $s >0$, the \emph{Sobolev space} $W^s(M)$ is defined by a completion of $C^\infty(M)$ equipped with the norm $\norm{\cdot}_s$. The Sobolev space $W^s(M) = \bigoplus_{n\in \Z} W^s(H_n)$ decomposes to closed $\mathsf{H}^g$-invariant subspaces $W^s(H_n) = W^s(M) \cap H_n$.

\subsubsection{Renormalization flow}\label{sec;renorm}
Denote diagonal matrix $\delta_i = diag(d_1, \cdots, d_g)$ with $d_i = 1$, $d_k = 0$ if $k \neq i$.  Then, for each $1 \leq i \leq g$, we denote $\hat\delta_i =  \begin{bmatrix} \delta_i & 0 \\  0 & -\delta_i\end{bmatrix} \in \mathfrak{sp}_{2g} = Lie(Sp_{2g}(\R))$.

Any such $\hat\delta_i$ generate one-parameter subgroups of automorphism (renormalization flow) $r_i^t := e^{t\hat\delta_i}$. We denote (rank $d$) renormalization actions $r_\textbf{t} := r_{i_1}^{t_1}\cdots r_{i_d}^{t_d}$ for $\textbf{t} = (t_1,\cdots, t_d)$ and $1 \leq i_1, \cdots, i_d \leq g$. We also write the corresponding automorphism
\begin{equation}\label{def;expmap}
\exp (\bm{t} \hat\delta(d)): (x,y,z) \mapsto (e^{\bm{t} \hat \delta}x,e^{\bm{-t} \hat \delta}y,z)
\end{equation}
of the Heisenberg group (see \S \ref{sec;renorma} for its use).

\subsection*{Main results}
One of the main objects in this paper is to construct finitely-additive measures defined on the space of rectangles on the Heisenberg manifold $M$. We state our results with   an overview of Bufetov functional.

\begin{definition}\label{def;rect}
For $(m,\textbf{T}) \in M \times \R_+^d$, denote the \emph{standard rectangle} for action $\PP$,
\begin{equation}\label{eqn;rect}
\Gamma^X_\textbf{T}(m) = \{\PP^{d,\alpha}_{\bm{t}}(m) \mid \bm{t} \in U(\textbf{T}) = [0,\textbf{T}^{(1)}]\times \cdots \times [0,\textbf{T}^{(d)}]  \}.
\end{equation}
\end{definition}

Let $\QQ_y^{d,Y} := \exp(y_1Y_1+\cdots+y_dY_d), \ y = (y_1, \cdots, y_d) \in \R^d$ be an action generated by elements $Y_i$ of standard basis. Let $\phi^Z_z:= \exp(zZ), \ z \in \R$ be a flow generated by the central element $Z$.
\begin{definition} Let $\mathfrak{R}$ be the collection of the \emph{generalized rectangles} in $M$. For any $1 \leq j \leq d \leq g$ and $\bm{t} = (t_1, \cdots, t_d)$, we set
\begin{equation}\label{def;genrec}
\mathfrak{R}:= \bigcup_{1 \leq i \leq d}\bigcup_{(y,z) \in \R^j\times \R}\bigcup_{(m,\textbf{T}) \in M \times \R_+^d}\{(\phi^Z_{t_i z})\circ\QQ_y^{j,Y}\circ\PP^{d,\alpha}_{\bm{t}}(m) \mid  \bm{t} \in U(\textbf{T}) \}.
\end{equation}
\end{definition}

\begin{theorem}\label{hi}
For any irreducible representation $H$, there exists a finitely-additive measure $\hat\beta_H(\Gamma) \in \C$ defined on every standard $d$-rectangle $\Gamma$ such that the following holds:
\begin{enumerate}
\item $(\emph{Additive property})$ For any decomposition of disjoint rectangles $\Gamma = \bigcup_{i=1}^n \Gamma_i$ or whose intersections have zero measure,
$$\hat\beta_H(\alpha,\Gamma) = \sum_{i=1}^n\hat\beta_H(\alpha,\Gamma_i).$$
\item $(\emph{Scaling property})$ For $\bm{t} \in \R^d$,
$$\hat\beta_H(r_{\bm{t}}(\alpha),\Gamma) = e^{-(t_1 + \cdots+  t_d)/2}\hat\beta_H(\alpha,\Gamma).$$

\item $(\emph{Invariance property})$ For any action $\QQ_\tau^{j,Y}$ generated by $Y_i$'s for $\tau \in \R_+^j$ and  $1 \leq j \leq d$,
$$\hat\beta_H(\alpha,(\QQ^{j,Y}_\tau)_*\Gamma) = \hat\beta_H(\alpha,\Gamma).$$

\item $(\emph{Bounded property})$ For any rectangle $\Gamma \in \mathfrak{R}$, there exists a constant $C(\Gamma) >0$ such that for $\hat X = \hat X_1\wedge \cdots \wedge \hat X_d$,
$$|\hat\beta_H(\alpha,\Gamma)| \leq C(\Gamma)(\int_{\Gamma}|\hat X| )^{d/2}.$$
\end{enumerate}
\end{theorem}

\begin{corollary}\label{cor;extend} The functional $\hat\beta_H$ defined on standard rectangle $\Gamma^X_\textbf{T}$ extends to  the class $\mathfrak{R}$.
\end{corollary}

\begin{remark}
By the additive property of functionals on rectangles, it is plausible to extend the shape of the domain to a general boundary by approximaion. However, this method may cause limitations on estimate by having a weaker bound for ergodic integrals. It may be also interesting to compare the methods of Shah \cite{S09a,S09b} and to obtain effective equidistribution results for smooth (or Lipschitz) boundary in our settings.
\end{remark}

Let us consider arbitrary two $d$-standard rectangles $U({\textbf{T}_1})$ and $U({\textbf{T}_2})$. For convenience, these rectangles are translated to intersect at only one vertex as in the figure.
Without loss of generality, assume that $d$ distinct faces of $U({\textbf{T}_1})$ emanate from the origin and $U({\textbf{T}_1}) \cap U({\textbf{T}_2}) = \{(\textbf{T}^{(l)}_1)_{1\leq l \leq d}\}$.
Denote by $\mathbf{P}(\textbf{T}_1)$ a collection of $2^d$ vertices $v = (v^{(1)},v^{(2)}, \cdots, v^{(d)})\in \R_+^d$ of $U(\textbf{T}_1)$.

Then we define a vector $\textbf{T}_{v} = (\textbf{T}^{(l)}_v) \in \R_+^d$ associated with $v$ given by
$$
\textbf{T}^{(l)}_v :=
\begin{cases}
\textbf{T}_{2}^{(l)} & \text{if} \quad v^{(l)} = \textbf{T}_{1}^{(l)}\\
\textbf{T}_{1}^{(l)} & \text{if} \quad  v^{(l)} = 0.
\end{cases}
$$

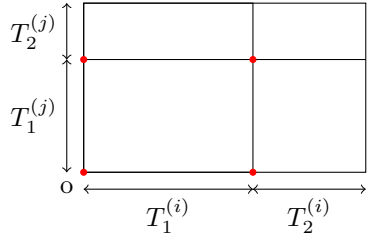
\begin{figure}\label{figg;fig2}\centering
\begin{tikzpicture}[scale=0.75]
\draw [draw=black] (5,3) rectangle (0,0);
\draw [draw=black] (3,3) rectangle (0,0);
\draw[draw=black] (0,2) -- (5,2);
   \draw[<->](0,-0.3)--(3,-0.3);
   \draw node[below] at (1.5,-0.3) {$T_{1}^{(i)}$};
   \draw node[below] at (4,-0.3) {};
      \draw[<->](-0.3,0)--(-0.3,2);
         \draw node[left] at (-0.3,1) {$T_{1}^{(j)}$};
         \draw node[left] at (-0.3,2.5) {$T^{(j)}_2$};
            \draw[<->](-0.3,2)--(-0.3,3);
                     \draw node[below] at (4,-0.3) {$T^{(i)}_2$};
               \draw[<->](3,-0.3)--(5,-0.3);
    \filldraw [red] (0,0) circle (1.5pt);
        \filldraw [red] (3,2) circle (1.5pt);
                \filldraw [red] (0,2) circle (1.5pt);
                        \filldraw [red] (3,0) circle (1.5pt);
                           \draw node[left][below] at (-0.3,0) {o};
\end{tikzpicture}
    \caption{Illustration of the rectangles $U({\textbf{T}_1})$ and $U({\textbf{T}_2})$ on $i,j$-th coordinate.}

\end{figure}

\begin{corollary}\label{hi2}  Let us denote
\begin{equation}\label{def;cocycle}
\beta_H(\alpha,m,\textbf{T}) := \hat\beta_H(\alpha,\Gamma^X_{\textbf{T}}(m)).
\end{equation}
The function $\beta_H$  satisfies the following properties:
\begin{enumerate}
\item \emph{(Cocycle property)} For all $(m,\textbf{T}_1,\textbf{T}_2) \in M \times \R_+^d \times \R_+^d$,
$$\beta_H(\alpha,m,\textbf{T}_1 + \textbf{T}_2) = \sum_{v \in \mathbf{P}(\textbf{T}_1)} \beta_H(\alpha,\PP_{v}^{d,\alpha}(m),\textbf{T}_v).$$
\item \emph{(Scaling property)} For all $m \in M$ and $\bm{t} = (t_1, \cdots, t_d)$,
$$\beta_H(r_{\bm{t}}(\alpha),m,\bm{T}) = e^{-(t_1 + \cdots + t_d)/2}\beta_H(\alpha,m,\bm{T}).$$

\item $(\emph{Bounded property})$ Let us denote the largest length of edges of $U(\bm{T})$ by  $T_{max} = \displaystyle\max_{1 \leq i \leq d}\textbf{T}^{(i)} $. Then there exists a constant $C >0$ such that
$$\beta_H(\alpha,m,\bm{T}) \leq C T_{max}^{d/2}.$$

\item \emph{(Orthogonality)} For all $\bm{T} \in \R^d$, $\beta_H(\alpha,\cdot,\bm{T})$ belongs to an irreducible component, i.e,
$$\beta_H(\alpha,\cdot,\bm{T}) \in H \subset L^2(M). $$
\end{enumerate}
\end{corollary}

By representation theory introduced in \S \ref{sec;repr}, for any $f = \sum_Hf_H \in W^s(M)$, define a \emph{Bufetov cocycle} associated to $f$  as a sum 
\begin{equation}\label{BBB}
\beta^f(\alpha,m,\bm{T})  = \sum_H D_\alpha^H(f)\beta_{H}(\alpha,m,\bm{T}).
\end{equation}

\emph{Notation.} Let $\hat X_i^\alpha$ be a  1-form  dual to the vector field $X_i^\alpha$, in the sense that
$\hat X_i^\alpha(X_i^\alpha) = 1$ and $\hat X_i^\alpha(X_j^\alpha) = 0$ if $i\neq j$ on $M$.
Given a Jordan region $U$ and a point $m \in M$, set  $\PPP^{d,\alpha}_Um$ the \emph{Birkhoff integrals} (currents) associated to the action $\mathsf{P}^{d,\alpha}_x$  given by
$$\left\langle \PPP^{d,\alpha}_Um , \omega_f \right\rangle := \int_U f(\mathsf{P}^{d,\alpha}_xm)dx_1\cdots dx_d, \quad \text{for } x \in \R^d,$$
for any degree $d$ $\mathfrak{p}$-form $\omega_f = f\hat X_1^\alpha \wedge \cdots  \wedge \hat X_d^\alpha$ with a smooth function $f \in C^\infty_0(M)$ with zero averages.

\begin{theorem}\label{6.2type}
For all $s >  d(d+11)/4+g+1$, there exists a constant $C_s >0$ such that
for almost all frequency $\alpha$, for all $f \in W^s(M)$ and for all $(m,\bm{T}) \in M \times \R^d$, we have
\begin{equation}\label{eqn:asymp}
|\left\langle \PPP^{d,\alpha}_{U(\bm{T})}m , \omega_f \right\rangle - \beta^f(\alpha,m,\bm{T}) | \leq C_s\norm{\omega_f}_{s}
\end{equation}
for $U(\bm{T}) = [0,T^{(1)}]\times \cdots \times [0,T^{(d)}]$ and $\omega_f = f\omega^{d,\alpha} \in  \Lambda^d\mathfrak{p} \otimes W^s(M)$.
\end{theorem}

The family of random variable
$$E_{\bm{T_n}}(f): = \frac{1}{{vol(U(\bm{T_n}))}^{1/2}}\left\langle \PPP^{d,\alpha}_{U(\bm{T_n})}(m) , \omega_f \right\rangle$$
is defined where $U(\bm{T_n})$ is a sequence of rectangles. The point $m \in M$ is distributed accordingly to the probability measure $vol$ on $M$.  
Our goal is to understand the asymptotic behavior of the probability distributions of $E_{\bm{T_n}}(f)$ as $U(\bm{T_n}) \nearrow \R^d$ in a sense of Følner.

\begin{theorem}\label{limit}
Let $\{\bm{T_n}\}$ be any sequence such that
$$\lim_{n\rightarrow \infty} r_{\log \bm{T_n}} [\alpha] = \alpha_\infty \in \M_g.$$
For every closed form $\omega_f \in \Lambda^d\mathfrak{p} \otimes W^s(M)$ with $s> d(d+11)/4+g+1$, which is not a coboundary, the limit distribution of the family of random variables $E_{\bm{T_n}}(f)$ exists along a subsequence of $\{\bm{T_n}\}$.
In particular, for almost all $\alpha$, the limit distribution of $E_{\bm{T_n}}(f)$ has compact support.
\end{theorem}

We finish the section by giving some remarks on a new adaptation of transfer operator techniques from hyperbolic dynamics. 
The method stems from the analysis of the transfer operator, firstly treated by P. Gieulietti and C.Liverani \cite{giulietti2019parabolic}. They set up a non-linear flow on the torus and proved asymptotics of ergodic averages with expansions of invariant distributions and eigenvalues of transfer operators called \emph{Ruelle resonances} (see also \cite{AB18, B19, faure2019ruelle, forni2020ruelle} for applications in parabolic flows).

A recent work of L. Simonelli and O. Butterley \cite{butterley2020parabolic} reproved some of the results of Flaminio and Forni \cite{FF06} by analytical methods of hyperbolic theory, not relying on representation theory.  Their work was restricted to a periodic type of flow, (the flow is only renormalized by partially hyperbolic diffeomorphism) but their methods showed indirect similarities with the work of Forni-Kanigowski \cite{forni2020time} and our current work.\footnote{For instance, it is plausible to view their construction of functionals obtained by spectral projection as a functional $\hat\beta_H$ on each sub-representation $H$ in our setting (see  \cite[\S 4]{butterley2020parabolic}).}
However, it is still not known if it is possible to extend their approach to full measure set of automorphism $\alpha$, requiring an existence of new renormalization cocycle, so called \emph{Transfer cocycle}. It is also not well-studied how to replace the use of Anisotropic norm, not relying on the previous result from Faure-Tsujii \cite{FT15}.\footnote{Furthermore, instead of relying on the methods of Sobolev constant techniques introduced in \cite{FF06}, we do not have proper tools on Anisotropic spaces yet.}

On higher step nilmanifolds, there does not exist a renormalization flow (moduli space $\M$ is trivial). It is not possible to construct the Bufetov functionals by the same strategies introduced in here, but other methods in handling non-renormalizable flows are possibly applied (cf. \cite{FF14,FFT15, kim2020effective}). \medskip

\emph{Outline of the paper.} In section \ref{sec;2}, we give basic definitions on higher rank actions, moduli spaces and Sobolev spaces. In section \ref{sec;3}, we state main theorems and prove constructions of Bufetov functionals with the properties. In section \ref{sec;4}, we prove asymptotic formula of ergodic integrals and limit theorems for normalized ergodic integrals.  In section \ref{sec;5},  we prove $L^2$-lower bound of the Bufetov functionals on transverse torus to the actions. In section \ref{sec;6}, analyticity of functional and extensions of domain are provided.
In section \ref{sec;7}, there exist measure estimates of functionals on the sets where values of functionals are small. This result only holds when frame $\alpha$ is of  bounded type.

\section{Analysis on Heisenberg manifolds}\label{sec;2}
This section reviews the definitions of Sobolev space, currents, representation, and renormalization flows on moduli space. 

\subsection{Sobolev space}
\subsubsection{Sobolev norm}
Given a basis $ (V_i)$ of the Lie algebra $\mathfrak{h}^g$, we write a new basis $((\alpha^{-1})_*V_i)$ for $\alpha \in Aut_0(\mathsf{H}^g)$. Similarly, denote by Laplacian $\Delta_\alpha = -\sum_i(\alpha^{-1})_*V_i^2$ with respect to the new basis. For any $s \in \R$ and any $f \in C^\infty(M)$, \emph{Sobolev norm} is defined by
$$\norm{f}_{\alpha,s} = \langle f, (1+\Delta_\alpha)^sf \rangle^{1/2}.$$
Let $W_\alpha^s(M)$ be a completion of $C^\infty(M)$ with the norm above. The dual space of $W_\alpha^s(M)$ is denoted by $W_\alpha^{-s}(M)$ and it is isomorphic to $W_\alpha^s(M)$. Extending it to the exterior algebra, define the Sobolev spaces of the form $\Lambda^d\mathfrak{p} \otimes W_\alpha^s(M)$ of cochains of degree $d$ and we use the same notations for the norm.

\subsubsection{Sobolev bundle}
The group $Sp_{2g}(\Z)$ acts on the right on the trivial bundles
$$Sp_{2g}(\R) \times W^s(M) \rightarrow Sp_{2g}(\R)$$
where
$$(\alpha,\varphi) \mapsto (\alpha,\varphi)\gamma = (\alpha \gamma,\gamma^*\varphi), \quad \gamma \in Sp_{2g}(\Z).$$

 We obtain the quotient flat bundle of Sobolev spaces over the moduli space:
$$(Sp_{2g}(\R) \times W^s(M)) / Sp_{2g}(\Z) \rightarrow \mathfrak{M}_g = Sp_{2g}(\R) / Sp_{2g}(\Z)$$
and the fiber over $[\alpha] \in \mathfrak{M}_g$ is locally identified with the space $W_\alpha^{s}(M)$.

By invariance of $Sp_{2g}(\Z)$ action, the class of $(\alpha,\varphi)$ is denoted by $[\alpha,\varphi]$ and $Sp_{2g}(\Z)$-invariant Sobolev norm is written by
$$  \norm{f}_{\alpha,s}:= \norm{[\alpha, f]}_s.$$

We denote the bundle of $\mathfrak{p}$-forms of degree $j$ of Sobolev order $s$ by $A^j(\mathfrak{p}, \M^s)$. The space of continuous linear functional on $A^j(\mathfrak{p},\M^s)$ will be called the \emph{space of currents} of dimension $j$ and denoted by $A_j(\mathfrak{p},\M^{-s})$.
There is  a flat bundle of (currents) distribution $A_j(\mathfrak{p}, \M^{-s})$ whose fiber over $[\alpha]$ is locally identified with the space $W_\alpha^{-s}(M)$. Likewise, the space of Sobolev currents $A_j(\mathfrak{p}, W^{-s}(M))$ of dimension $j$ with order $s$  is identified with $\Lambda^j\mathfrak{p} \otimes W^{-s}(M)$.
We write the norm for form $\omega$ and currents $\DD$ by
$$  \norm{\omega}_{\alpha,s}:= \norm{[\alpha, \omega]}_s, \quad   \norm{\DD}_{\alpha,s}:= \norm{[\alpha, \DD]}_s.$$

\subsubsection{Best Sobolev constant}
The \emph{Sobolev embedding theorem} implies that for any $\alpha \in Sp_{2g}(\R)$ and $s > g+1/2$, there exists a constant $B_s(\alpha)>0$ such that
for any $f \in W^s_\alpha(M)$,
\[ \norm{f}_\infty \leq B_s(\alpha)\norm{f}_{\alpha,s}.
\]
The \emph{best Sobolev constant} is defined as the function on the $Sp_{2g}(\R)$ given by
\begin{equation}\label{def;sobolev}
B_s(\alpha) := \sup_{f \in W^s_\alpha(M) \backslash \{0\}} \frac{\norm{f}_\infty}{\norm{f}_{\alpha,s}}.
\end{equation}
By Lemma 4.4 in \cite{CF15}, the {best Sobolev constant} $B_s$ is a $Sp_{2g}(\Z)$-modular function on $\mathfrak{H}_g$. Thus, we shall write $B_s([[\alpha]])$ or $B_s([\alpha])$ for $B_s(\alpha)$.
By the Sobolev embedding theorem, we have a bound for the Birkhoff integral current. 
\begin{lemma}\label{lem;cf5.5}\cite[Lemma 5.5]{CF15}\label{3.1} For any Jordan region $U \subset \R^d$ with Lebesgue measure $|U|$, for any $s > g + 1/2$ and all $m \in M$,
$$\norm{[\alpha, \PPP_U^{d,\alpha}m]}_{-s} \leq B_s ([[\alpha]])|U|.$$
\end{lemma}

\subsection{Invariant currents}
\subsubsection{Identification}
 The \emph{boundary operators}
 $$\partial : A_j(\mathfrak{p},W^{-s}(M)) \rightarrow A_{j-1}(\mathfrak{p},W^{-s}(M))$$
  are adjoint of the differentials $d$ such that $\langle \partial \D, \omega \rangle  = \langle \D, d\omega \rangle$ for any  $\omega \in \Lambda^{j-1}\mathfrak{p} \otimes W^s(M)$. A current $\D$ is called \emph{closed} if $\partial \D = 0$.

  For $s>0$, we denote $Z_d(\mathfrak{p},W^{-s}(M))$ by the space of closed currents of dimension $d$ and Sobolev order $s$.
$I_d(\pp,W^{-s}(M))$ is denoted by the  $d$-dimensional space of $\mathsf{P}$-invariant currents of Sobolev order $s$. Now we review the relations between these two currents.

\begin{proposition}\cite[Proposition 3.13]{CF15} \label{prop;3.13}
For any $s > d/2 = \dim{\mathsf{P}}/2$, we have $I_d(\pp,\mathscr{S}(\R^g)) \subset W^{-d/2-\epsilon}(\R^g)$ for all $\epsilon>0$. Additionally,
\begin{itemize}
\item $I_d(\pp,\mathscr{S}(\R^g))$ is one dimensional space if $\dim \mathsf{P} = g$,
\item $I_d(\pp,\mathscr{S}(\R^g))$ is an infinite-dimensional space if $\dim \mathsf{P} < g$,
\item  $I_d(\pp,\mathscr{S}(\R^g)) = Z_d(\mathfrak{p},W^{-s}(M))$ for any $1 \leq d \leq g$.
\end{itemize}
\end{proposition}

\subsubsection{Basic currents}
 The current $B_\alpha$ is called \emph{basic} if for all $j \in \{i_1, \cdots,  i_d\}$,
$$\iota_{X_{j}}B_\alpha = L_{X_{j}}B_\alpha = 0.$$
For an irreducible representation $H$, there exists a unique \emph{basic current} $B_\alpha^H$ (of degree $2g+1-d$ and dimension $d$) associated to an invariant distribution $D_\alpha^H$. It is defined by $B_\alpha^H = D_\alpha^H \eta_X$ and
this formula implies that for every $d$-form $\xi$,
\begin{equation}\label{eqn;dual.basic}
B_\alpha^H(\xi) = D_\alpha^H \left(\frac{\eta_X \wedge \xi}{\omega_{vol}}\right)
\end{equation}
where $\eta_X := \iota_{X_{i_1}}\cdots\iota_{X_{i_d}} \omega_{vol}$ and $\omega_{vol}$ is an invariant volume form (with total unit volume).
\medskip

The basic current $B_\alpha^H$ belongs to the Sobolev space of currents and it is $\mathsf{P}$-invariant.
It follows that for all $s > d/ 2$, by Sobolev embedding theorem, $B_\alpha^H \in A_j(\mathfrak{p}, W_\alpha^{-s}(M))$ if and only if $D_\alpha^H \in W^{-s}_\alpha(M)$ for all $s > d/2$.

\begin{remark}The formula (\ref{eqn;dual.basic}) yields an isomorphism between the space of basic (closed) currents and invariant distributions (see also \cite[\S 6.1]{forni2002deviation} and \cite[\S 2.1]{BF14}). 
 For any $d$-dimensional $\pp$-form $\omega_f = f \hat X_1 \wedge \cdots  \wedge  \hat X_d$ with $f \in C^\infty(M)$, we will identify  currents $\D$ with distribution $D$ by writing
$$ \langle D,f\rangle= \langle \D, \omega_f\rangle.$$
\end{remark}

\subsubsection{Renormalization}\label{sec;renorma}
Let $s > d/2$. Recall the definition of renormalization flow \eqref{def;expmap} in \S \ref{sec;renorm}. For $\omega \in \Lambda^d\mathfrak{p} \otimes W^s(\R^g)$, $\D \in Z_d(\mathfrak{p},W^{-s}(M))$, and  $t \in \R$
$$r_i^t[\alpha,\omega] = [r_i^t\alpha, \omega], \quad  r_i^t[\alpha, \D] = [r_i^t\alpha, \D].$$

By Proposition 5.2 in \cite{CF15}, the sub-bundle $Z_d(\mathfrak{p},W^{-s}(M))$ is invariant under the renormalization flows $r_i^t$. Furthermore, we have
\begin{equation}\label{eqn;interntwine}
\norm{r_1^{t_1}\dotsc r_d^{t_d}[\alpha, \D]}_{-s} = e^{-(t_1+\cdots + t_d)/2}\norm{[\alpha, \D]}_{-s}.
\end{equation}

By reparametrization of $\R^d$-action $r_{\bm{t}} = r_1^{t_1}\dotsc r_d^{t_d}$,
\begin{equation}\label{eqn;reno1}
\PP_x^{d, (r_1^{t_1}\dotsc r_d^{t_d}\alpha)} = \PP_{(e^{-t_1}x_1,...,e^{-t_d}x_d)}^{d, \alpha}.
\end{equation}
Then, denoting $(e^{-t_1},...,e^{-t_d})U$ diagonal automorphisms of $\R^d$ applied to $U$, the Birkhoff integral current also satisfies the following identity
\begin{equation}\label{eqn;birkhoff}
\PPP_U^{d, (r_1^{t_1}...r_d^{t_d}\alpha)}m = e^{(t_1+\cdots + t_d)}\PPP_{(e^{-t_1},...,e^{-t_d})U}^{d, \alpha}m.
\end{equation}

Let $U_\textbf{t} : L^2(\R^d) \rightarrow  L^2(\R^d) $ be a unitary operator for $\textbf{t} = (t_1,\cdots, t_d)$,
\begin{equation}\label{eqn;interntwining}
U_\textbf{t}f(x) = e^{(t_1+\cdots + t_d)/2}f(e^{t_1}x_1, \cdots, e^{t_d}x_d)
\end{equation}
for $x \in \R^d$. This map intertwines $\alpha=  (X,Y,Z)$ and $r_\textbf{t}(\alpha) = (r_{\bm{t}}X, r_{-\bm{t}}Y,Z)$. I.e their derived representation \eqref{def;derived} intertwines by $U_\textbf{t}$. Here we note that the domain of operator $U_\textbf{t}$ trivially extends by adding non-scaled coordinates to $L^2(\R^g)$ if $d <g$.

\subsection{Diophantine condition }

\subsubsection{Height function}
\begin{definition}[Height function]
The \emph{height} of a point $Z \in \mathfrak{H}_g$ in Siegel upper half space is the positive number
$$hgt(Z):= \text{det} \Im(Z).$$
The \emph{maximal height function} $Hgt:\Sigma_g \rightarrow \R^+$ is the maximal height of a $Sp_{2g}(\Z)$ orbit of $Z$. That is, for the class of $[Z] \in \Sigma_g$,
$$Hgt([Z]) := \max_{\gamma \in Sp_{2g}(\Z)}hgt(\gamma(Z)). $$
\end{definition}

By Proposition 4.8 of \cite{CF15}, there exists a universal constant $C(s)>0$ such that the best Sobolev constant satisfies the estimate
\begin{equation}\label{ineq;sobolev}
B_s([[\alpha]]) \leq C(s)\cdot (Hgt[[\alpha]])^{1/4}.
\end{equation}

We rephrase Lemma 4.9 in \cite{CF15} regarding the bound of renormalized height.
\begin{lemma}
For any $[\alpha] \in \M_g$ and any $\textbf{t} \in \R_+^d$,
\begin{equation}\label{201}
Hgt([[\exp (-\bm{t} \hat\delta(d))\alpha]]) \leq (\text{det}(e^{\bm{t} \hat \delta}))^2Hgt([[\alpha]]).
\end{equation}
\end{lemma}

\subsubsection{Diophantine condition}\label{sec;dio}
\begin{definition}
An automorphism $\alpha \in Sp_{2g}(\R)$ or a point $[\alpha] \in \M_g$ is \emph{$\hat\delta(d)$-Diophantine of type $\sigma$} if there exists a $\sigma>0$ and a constant $C>0$ such that
\begin{equation}
{Hgt}([[\exp (-\bm{t} \hat\delta(d))\alpha]]) \leq CHgt([[\exp (-\bm{t} \hat\delta(d))]])^{(1-\sigma)}{Hgt}([[\alpha]]), \ \forall \bm{t} \in \R_+^d.
\end{equation}
This states that $\alpha \in Sp_{2g}(\R)$ satisfies $\hat\delta(d)$-Diophantine if the height of the projection of $\exp (-\bm{t} \hat\delta(d))\alpha$ in the Siegel modular variety $\Sigma_g$ is bounded by $e^{2(t_1+\cdots +t_d)(1-\sigma)}$. Furthermore,
\begin{itemize}
\item $[\alpha] \in \M_g$ satisfies a \emph{$\hat\delta(d)$-Roth condition} if for any $\epsilon>0$ there exists a constant $C>0$ such that
\begin{equation}\label{eqn;roth}
{Hgt}([[\exp (-\bm{t} \hat\delta(d))\alpha]]) \leq CHgt([[\exp (-\bm{t} \hat\delta(d))]])^{\epsilon}{Hgt}([[\alpha]]), \ \forall \bm{t} \in \R_+^d.
\end{equation}
That is, $\hat\delta(d)$-Diophantine of {every} type $0<\sigma <1$.

\item $[\alpha]$ is of \emph{bounded type} if there exists a constant $C>0$ such  that
\begin{equation}\label{bdd}
Hgt([[\exp (-\bm{t} \hat\delta(d))\alpha]]) \leq C, \ \forall \bm{t} \in \R_+^d.
\end{equation}
\end{itemize}
\end{definition}

\begin{definition}
Let $X = G/\Lambda$ be a homogeneous space  equipped with the probability
Haar measure $\mu$. A function $\phi : X \rightarrow \R$ is said $k$-DL (distance like)
for some exponent $k > 0$ if it is uniformly continuous and if there exist
constants $C_1, C_2 > 0$ such that
$$C_1e^{-kz} \leq \mu(\{x \in X \mid \phi(x) \geq z \})\leq C_2e^{-kz}, \quad \forall z \in \R.$$
\end{definition}

By the work of Kleinblock and Margulis, a multi-parameter generalization of Khinchin-Sullivan logarithm law for geodesic excursion \cite{Sul82} holds.
\begin{theorem}\cite[Theorem 1.9]{KM99}\label{prop;cartan1}
Let $G$ be a connected semisimple Lie group without compact factors, $\mu$ its normalized Haar measure, $\Lambda \subset G$
an irreducible lattice, $\mathfrak{a}$ a Cartan subalgebra of the Lie algebra of G. Let $\mathfrak{d}_+$ be a non-empty open cone in a $d$-dimensional subalgebra $\mathfrak{d}$ of $\mathfrak{a}$ $(1 \leq d \leq rank_{\R}(G))$.
If $\phi : G/\Lambda \rightarrow \R$ is a $k$-DL function for some $k > 0$,
then for $\mu$-almost all $x \in G/\Lambda$ one has
$$\limsup_{\mathbf{z} \in \mathfrak{d}_+, \mathbf{z}\rightarrow \infty}\frac{\phi(\exp(\mathbf{z}) x)}{\log\norm{\mathbf{z}}} = \frac{d}{k}$$
\end{theorem}

By Lemma 4.7 of \cite{CF15}, the logarithm of Height function is DL-function with exponent $k = \frac{g+1}{2}$ on the Siegel variety $\Sigma_g$ (and induces on $\M_g = Sp_{2g}(\R)/Sp_{2g}(\Z)$). Hence, we obtain the following proposition.

\begin{proposition}\label{prop;jackma}Under the assumption $X = \M_g$ of Theorem \ref{prop;cartan1}, for $s> g+ 1/2$, there exists a full measure set $\Omega_g(\hat\delta)$ and for all $[\alpha] \in \Omega_g(\hat\delta) \subset  \M_g$   
\begin{equation}\label{jackma}
\limsup_{\bm{t} \in \R^{d}_+, \bm{t} \rightarrow \infty} \frac{\log Hgt ([[\exp (-\bm{t} \hat\delta(d))\alpha]]) }{\log \norm{\bm{t}}} \leq \frac{2d}{g+1}.
\end{equation}
Any such $[\alpha]$ satisfies the $\hat\delta(d)$-Roth condition \eqref{eqn;roth}.
\end{proposition}

For any $L>0$ and $1\leq d \leq g$, let $DC(L)$ denote the set of  $[\alpha] \in \M_g$ such that
\begin{equation}\label{cond;newdc}
\int_0^\infty\cdots\int_0^\infty e^{-(t_1+ \cdots+ t_d)/2}\text{Hgt} ([[r_\textbf{-t}(\alpha)]])^{1/4}dt_1\cdots dt_d  \leq L.
\end{equation}
Let $DC$ denote the union of the sets $DC(L)$ over all $L > 0$.
It follows immediately that the set $DC \subset \M_g$ has full Haar volume.

\begin{remark} In the work of Cosentino-Flaminio \cite{CF15}, Diophantine condition is restricted to one-parameter subgroup $\exp (-t \hat\delta)$ for non-negative ($g$-dimensional) diagonal directions. This is based on the easy part of theorem in Kleinblock and Margulis \cite[Theorem 1.7]{KM99}. Our condition is for $d$-dimensional renormalization actions. 
We will write the notation $r_\textbf{-t}$ instead of $\exp (-\bm{t} \hat\delta(d))$ for later sections.
\end{remark}

\section{Constructions of the functionals}\label{sec;3}
In this section, we construct Bufetov functionals for higher rank actions on the standard rectanglular domain. 
Then we verify the properties of functionals $\hat\beta_H$ and cocycles $\beta_H$ stated in Theorem \ref{hi}.

\subsection{Remainder estimates}
For any exponent $s>d/2$, Hilbert bundle induces an orthogonal decomposition
$$A_d(\mathfrak{p}, \M^{-s}) = Z_d(\mathfrak{p}, \M^{-s}) \oplus R_d(\mathfrak{p}, \M^{-s})$$
where $R_d(\mathfrak{p}, \M^{-s}) = Z_d(\mathfrak{p}, \M^{-s})^\perp$.
Denote by $\I^{-s}$ and $\RR^{-s}$ the corresponding orthogonal projection operator and by $\I_\alpha^{-s}$ and $\RR_\alpha^{-s}$ the restrictions to the fiber over $[\alpha] \in \M$ for $\alpha \in Sp_{2g}(\R)$. In particular, for the current (Birkhoff integrals) $\DD = \PPP_{U}^{d,\alpha}m$, we  call $\I_\alpha^{-s}(\DD) = \I^{-s}[\alpha,\DD]$ \emph{boundary} term and $\RR_\alpha^{-s}(\DD) = \RR^{-s}[\alpha,\DD]$ \emph{remainder} term respectively. Consider the orthogonal decomposition of current
\begin{equation}\label{eqn;decom}
\DD = \I_{r_{-\bm{t}}[\alpha]}^{-s}(\DD) + \RR_{r_{-\bm{t}}[\alpha]}^{-s}(\DD).
\end{equation}

We firstly recall the following estimate of boundary term.
\begin{lemma}\label{lem;boundary1} \cite[Lemma 5.7]{CF15} Let $s > d/2+2$. There exists a constant $C = C(s)>0$ such that for all $t_i \geq 0$ for $1\leq i \leq d$, we have
\begin{multline*}
\norm{\I^{-s}[\alpha, \DD]}_{-s} \leq e^{-(t_1+\cdots + t_d)/2}\norm{\I^{-s}[r_1^{-t_1}\cdots r_d^{-t_d}\alpha, \DD]}_{-s}\\
 + C_1|t_1+ \cdots + t_d| \int_{0}^1e^{-u(t_1 + \cdots + t_d)/2}\norm{\RR^{-s}[r_1^{-ut_1}\cdots r_d^{-ut_d}\alpha, \DD}_{-(s-2)}du.
 \end{multline*}
\end{lemma}

 By Stokes' theorem, we have the following remainder estimate. 
\begin{lemma}\cite[Lemma 5.6]{CF15}\label{eqn;stokes} Let $s > g+d/2+1$. For any non-negative $s' < s- (d+1)/2$ and Jordan region $U \subset \R^d$, there exists $C = C(d,g,s,s')>0$ such that
$$\norm{\RR^{-s}[\alpha, (\PPP_{U}^{d,\alpha}m)]}_{-s} \leq C \norm{[\alpha, \partial(\PPP_{U}^{d,\alpha}m)]}_{-s'}.$$
\end{lemma}

A quantitative bound of Birkhoff integrals on the square domain was obtained in \cite{CF15}, but we need to extend the result to the rectangular shapes for analyticity of functionals in the section \S \ref{sec;6}.

From now on, assume that  $s_{d,g}:= d(d+11)/4+g+1/2$.

\begin{theorem}\label{thm;3.3} For $s > s_{d,g}$, there exists a constant $C = C(s,d) >0$ such that the following holds.
 For any $t_i >0$, $m \in M$ and ${U_d(t)} = [0,e^{t_1}]\times \cdots \times [0,e^{t_d}]$, we have
\begin{align}\label{eqn;main2}
\begin{split}
\norm{[\alpha, (\PPP_{U_d(t)}^{d,\alpha}m)]}_{-s} &\leq C\sum_{k=0}^d \sum_{1 \leq i_1 < \cdots < i_k \leq d} \int_{0}^{t_{i_k}}\cdots\int_{0}^{t_{i_1}} \exp({\frac{1}{2}\sum_{l=1}^d t_l - \frac{1}{2}\sum_{l=1}^k u_{i_l} })\\
& \times  Hgt ( [[\prod_{1\leq j \leq d}r_j^{-t_{j}}\prod_{l=1}^k r_{i_l}^{u_{i_l}}\alpha]] )^{1/4}du_{i_1}\cdots du_{i_k}.
 \end{split}
 \end{align}
\end{theorem}
\begin{proof} We proceed by induction. For $d = 1$, it follows from the Theorem 5.8 in \cite{CF15}. We assume that the result holds for $d-1$.
Decompose the current as a sum of boundary and remainder term as in \eqref{eqn;decom}.

\vspace{4pt}\noindent
\textbf{Step 1.} We firstly estimate the boundary term. By Lemma \ref{lem;boundary1}, renormalize the terms with $r^u = r_1^u\cdots r_d^u$. Then, we have 
\begin{align}\label{ine;step1}
\begin{split}
\norm{\I^{-s}[\alpha, (\PPP_{U_d(t)}^{d,\alpha}m)]}_{-s} &\leq e^{-(t_1+\cdots + t_d)/2}\norm{\I^{-s}[r_1^{-t_1}\cdots r_d^{-t_d}\alpha, (\PPP_{U_d(t)}^{d,\alpha}m)]}_{-s}\\
& + C_1(s) \int_{0}^{t_1+\cdots + t_d}e^{-ud/2}\norm{\RR^{-s}[r^{-u}\alpha, (\PPP_{U_d(t)}^{d,\alpha}m)]}_{-(s-2)}du\\
& := (I) + (II).
\end{split}
 \end{align}

By renormalization \eqref{eqn;birkhoff} and Lemma \ref{lem;cf5.5} for unit volume,
\begin{align*}
& \norm{\I^{-s}[r_1^{-t_1}\cdots r_d^{-t_d}\alpha, (\PPP_{U_d(t)}^{d,\alpha}m)]}_{-s} \\
 &= e^{t_1+\cdots + t_d}\norm{\I^{-s}[r_1^{-t_1}\cdots r_d^{-t_d}\alpha, (\PPP_{U_d(0)}^{d,r_1^{-t_1}\cdots r_d^{-t_d}\alpha}m)]}_{-s}\\
& \leq C_2e^{t_1+\cdots + t_d}Hgt ([[r_1^{-t_1}\cdots r_d^{-t_d}\alpha]] )^{1/4}.
 \end{align*}

Hence
$$ (I) \leq C_2e^{(t_1+\cdots t_d)/2}Hgt ([[r_1^{-t_1}\cdots r_d^{-t_d}\alpha]] )^{1/4},$$
where the sum corresponds to the first term $(k=0)$ in the statement.

\vspace{4pt}\noindent
\textbf{Step 2.} To estimate $(II)$, in view of \eqref{eqn;birkhoff} and Lemma \ref{eqn;stokes} provided $s' < (s-2) - (d+1)/2$, we have
\begin{align}
\label{ine;step2}
\begin{split}
\norm{\RR^{-s}[r^{-u}\alpha, (\PPP_{U_d(t)}^{d,\alpha}m)]}_{-(s-2)} &= \norm{e^{ud}\RR^{-s}[r^{-u}\alpha, (\PPP_{U_d(t-u)}^{d,r^{-u}\alpha}m)]}_{-(s-2)}\\
& \leq C_3(s,s')e^{ud}\norm{[r^{-u}\alpha, \partial(\PPP_{U_d(t-u)}^{d,r^{-u}\alpha}m)]}_{-s'}.
\end{split}
 \end{align}

 The boundary  $\partial (\PP_{U_d}^{d, r^{-t}\alpha})$ is the sum of $2d$ currents of dimension $d-1$. These currents are Birkhoff sums of $d$ face subgroups obtained from $\PP_{U_d}^{d, r^{-t}\alpha}$ by omitting one of the base vector fields $X_i$. It is reduced to $(d-1)$ dimensional shape obtained from $U_d(t-u):= [0,e^{t_1-u}]\times \cdots \times[0,e^{t_d-u}]$. For each $1\leq j \leq d$, there are Birkhoff sums along $d-1$ dimensional cubes.
 By induction hypothesis, we add all the $d-1$ dimensional cubes by adding all the terms along $j$:
\begin{align}\label{ine;norm}
\begin{split}
& \norm{[r^{-u}\alpha, (\PPP_{U_{d-1}(t-u)}^{d-1,r^{-u}\alpha}m)]}_{-s'}   \leq  C_4(s',d-1)\sum_{j=1}^{d}\sum_{k=0}^{d-1} \sum_{\substack{1 \leq i_1 < \cdots < i_k \leq d\\ i_l \neq j, \forall l} }\int_{0}^{t_{i_k-u}}\cdots\int_{0}^{t_{i_1-u}}\\
& \exp({\frac{1}{2}\sum_{\substack{l=1\\ l\neq j }}^d (t_l - u)   - \frac{1}{2}\sum_{l=1}^k u_{i_l} }) Hgt ( [[\prod_{\substack{1\leq l \leq d \\  l \neq j }}r_l^{-(t_{l}-u)}\prod_{l=1}^k r_{i_l}^{u_{i_l}}(r^{-u}\alpha)]] )^{1/4}du_{i_1}\cdots du_{i_k}.
\end{split}
\end{align}

Combining \eqref{ine;step1} and \eqref{ine;step2}, we obtain the estimate for $(II)$.
\begin{align*}
(II) & \leq C_5(s',d-1)\sum_{j=1}^{d}\sum_{k=0}^{d-1} \sum_{\substack{1 \leq i_1 < \cdots < i_k \leq d\\ i_l \neq j} } \int_{0}^{t_1 + \cdots t_d}\int_{0}^{t_{i_k-u}}\cdots\int_{0}^{t_{i_1-u}}du_{i_1}\cdots du_{i_k}du\\
&   \times \exp({\frac{1}{2}\sum_{\substack{l=1\\ l\neq j }}^d t_l  - \frac{1}{2}u  - \frac{1}{2}\sum_{l=1}^k u_{i_l} })Hgt ( [[\prod_{1\leq l \leq d}r_l^{-t_{l}}\prod_{l=1}^k r_{i_l}^{u_{i_l}}(r_j^{-u+t_j}\alpha)]] )^{1/4}.
 \end{align*}
Applying the change of variable $u_{j} = t_{j} - u$, we obtain
\begin{align*}
&(II)  \leq C_6(s',d-1)\sum_{j=1}^{d}\sum_{k=1}^{d-1} \sum_{\substack{1 \leq i_1 < \cdots < i_k \leq d\\ i_l \neq j} } \\
& \times \Big(\int_{-(t_1 + \cdots t_d) + t_j}^{t_j}\int_{0}^{t_{i_k-u}}\cdots\int_{0}^{t_{i_1-u}}du_{i_1}\cdots du_{i_k}du_j\\
&   \times \exp\big({\frac{1}{2}(t_1 + \cdots  + t_d) - \frac{1}{2}u_j  - \frac{1}{2}\sum_{l=1}^k u_{i_l} }\big)Hgt ( [[\prod_{1\leq l \leq d}r_l^{-t_{l}}\prod_{l=1}^k r_{i_l}^{u_{i_l}}(r_j^{-u_{j}}\alpha)]] )^{1/4}\Big).
 \end{align*}
Simplifying multi-summation above, (with $-(t_1 + \cdots t_d) + t_j \leq 0$)
\begin{align*}
(II) & \leq C_7(s',d)\sum_{k=1}^{d} \sum_{\substack{1 \leq i_1 < \cdots < i_k \leq d} } \int_{0}^{t_{i_k}}\cdots\int_{0}^{t_{i_1}}du_{i_1}\cdots du_{i_k}\\
&   \times \exp\big({\frac{1}{2}(t_1 + \cdots  t_d)   - \frac{1}{2}\sum_{l=1}^k u_{i_l} }\big)Hgt ( [[\prod_{1\leq l \leq d}r_l^{-t_{l}}\prod_{l=1}^k r_{i_l}^{u_{i_l}}\alpha]] )^{1/4}.
 \end{align*}

\noindent
\textbf{Step 3.}  Now we turn to estimate of remainder terms. By Lemma \ref{eqn;stokes}, the remainder term of $d$-dimensional rectangles (cube) decomposes to lower dimensional boundary and remainder terms (up to dimension 1). Combining with the estimate from step 1 and 2, we have the following
\begin{equation}\label{eqn;remain6}
\norm{\RR^{-s}[\alpha, \partial(\PPP_{U_d}^{d,\alpha}m)]}_{-s} \leq C(s)\sum_{i=1}^{d-1} \norm{\I^{-s}[\alpha, (\PPP_{U_i}^{i,\alpha}m)]}_{-s} + \norm{\RR^{-s}[\alpha, (\PPP_{U_1}^{1,\alpha}m)]}_{-s}
\end{equation}
where $U_i$ is $i$-dimensional rectangle. The sum of the boundary terms is absorbed in the bound of $(I) + (II)$. For 1-dimensional remainder with interval $\Gamma_T$, the boundary is a 0-dimensional current. Then,
$$\langle \partial(\PPP_{[0,T]}^{1,\alpha}m), f \rangle = f(\PP_{T}^{1,\alpha}m) - f(m).$$
Hence, by Sobolev embedding theorem and by definition of Sobolev constant \eqref{def;sobolev} and  \eqref{ineq;sobolev},
$$\norm{\RR^{-s}[\alpha, \partial(\PPP_{[0,T]}^{1,\alpha}m)]}_{-s} \leq 2B_{s}([[\alpha]] ) \leq C(s)Hgt ([[\alpha]] )^{1/4}.$$
Then, by inequality (\ref{201})
\begin{align*}
C(s)Hgt ([[\alpha]] )^{1/4} &= C(s)Hgt ([[r_{\bm{t}}r_{-\bm{t}}\alpha]] )^{1/4}\\
&\leq C(s)e^{(t_1+\cdots +t_d)/2}Hgt ([[r_1^{-t_1}\cdots r_d^{-t_d}\alpha]] )^{1/4}.
 \end{align*}
This implies that 1-dimensional remainder has the same type of bound $(I)$.
Therefore, the theorem follows from combining all the terms $(I), (II)$, and $d$-dimensional remainder \eqref{eqn;remain6}.
\end{proof}

\begin{remark}
It is also reasonable to decompose the rectangle $U_d(t)$ as a union of several squares and renormalize their faces. However, it may involve computational difficulties. Plus, an approach of simply summing up squares may provide a weaker upper bound of ergodic integrals. To obtain a necessary bound for the Lemma \ref{04}, we rather simply generalized the strategy of Theorem 5.10 in \cite{CF15}.
\end{remark}

Let us set
\begin{equation}\label{def;kts}
\K_{\alpha,\textbf{t},s}(\Gamma)  := \norm{\RR^{-s}[r_{-\bm{t}}(\alpha), (\PPP_{U_\Gamma}^{d,r_{-\bm{t}}(\alpha)}m)]}_{-(s+1)}.
\end{equation}
Now we prove the remainder estimate that will be used in \eqref{eqn;remainders}.
\begin{lemma}\label{04} Let $s > s_{d,g}$. There exists a constant $C(s,\Gamma)>0$ such that for any rectangle $U_\Gamma =  [0,e^{\Gamma_1}]\times \cdots \times [0,e^{\Gamma_d}]$,
$$\K_{\alpha,\textbf{t},s}(\Gamma) \leq  C(s,\Gamma)Hgt ([[r_{-\bm{t}}(\alpha)]] )^{1/4}.$$
\end{lemma}
\begin{proof}
By Lemma \ref{eqn;stokes}, we estimate the bound of $d-1$ (renormalized) currents instead of $U_\Gamma$. Then, by Theorem \ref{thm;3.3},
we obtain 
\begin{align*}
\K_{\alpha,t}(\Gamma) & \leq C\sum_{k=0}^{d-1} \sum_{1 \leq i_1 < \cdots < i_k \leq d-1} \int_{0}^{\Gamma_{i_k}}\cdots\int_{0}^{\Gamma_{i_1}} \exp({\frac{1}{2}\sum_{l=1}^{d-1} \Gamma_l - \frac{1}{2}\sum_{l=1}^k u_{i_l} })\\
& \times  Hgt ( [[\prod_{1\leq j \leq d-1}r_j^{-\Gamma_{j}}\prod_{l=1}^k r_{i_l}^{u_{i_l}}(r_{-\bm{t}}(\alpha))]] )^{1/4}du_{i_1}\cdots du_{i_k}.
\end{align*}
In view of (\ref{201}) for $0 \leq k \leq d-1$,
$$ Hgt ( [[\prod_{1\leq j \leq d-1}r_j^{-\Gamma_{j}}\prod_{l=1}^k r_{i_l}^{u_{i_l}}(r_{-\bm{t}}(\alpha))]] )^{1/4} \leq e^{\frac{1}{2}(\sum_{l=1}^k{u_{i_l}}  - \sum_{l=1}^d\Gamma_l)}Hgt ([[r_{-\bm{t}}(\alpha)]] )^{1/4}. $$

Then, we obtain 
\begin{align}\label{eqn;remainder1}
\begin{split}
&\norm{\RR^{-s}[r_{-\bm{t}}(\alpha), (\PPP_{U_\Gamma}^{d,r_{-\bm{t}}\alpha}m)]}_{-s} \\
&\leq  C(s) \Big(\sum_{k=0}^{d-1} \sum_{1 \leq i_1 < \cdots < i_k \leq d-1} \prod_{l=1}^k \Gamma_{i_l} \Big) Hgt ([[r_{-\bm{t}}(\alpha)]] )^{1/4}\\
& \leq C(s,\Gamma)Hgt ([[r_{-\bm{t}}(\alpha)]] )^{1/4}.
\end{split}
\end{align}
Therefore, we obtain the conclusion.
\end{proof}

\subsection{Constructions of the functionals}
For fixed $\alpha \in Aut_0(\mathsf{H}^g)$, let $\Pi^{-s}_H : A_d(\mathfrak{p}, W_\alpha^{-s}(M)) \rightarrow A_d(\mathfrak{p}, W_\alpha^{-s}(H))$ denote the orthogonal projection on a single irreducible unitary representation $H$. We further decompose this projection operator with the basic current $B^{-s,H}_{\alpha}$ and its remainder $R^{-s,H}_{\alpha}$ given by
$$\Pi^{-s}_H =  \B^{-s}_{H,\alpha}(\Gamma)B^{-s,H}_{\alpha} + R_\alpha^{-s,H}.$$
The map $\B^{-s}_{H,\alpha} : A_d(\mathfrak{p}, W_\alpha^{-s}(M)) \rightarrow \C$ denotes the orthogonal component in the direction of basic current, supported on a single irreducible unitary representation.

The {Bufetov functionals} on the standard rectangle $\Gamma$ are defined for all $[\alpha] \in DC$ as follows.

\begin{lemma}\label{35} Let $[\alpha] \in DC(L)$. For $s> s_{d,g}$,  the limit
$$\hat\beta_H(\alpha,\Gamma) = \lim_{t_1, \dotsc, t_d \rightarrow \infty} e^{-(t_1+ \cdots + t_d)/2}\B^{-s}_{H,r_\textbf{-t}(\alpha)}(\Gamma)$$
exists and it defines a finitely-additive measure on the set of standard rectangles.
Moreover, there exists a constant $C(s, \Gamma)>0$ such that the following estimate holds:
\begin{equation}\label{asym}
\norm{\Pi_{H,\alpha}^{-s}(\Gamma) - \hat\beta_H(\alpha,\Gamma)B_\alpha^H}_{\alpha,-s} \leq  C(s,\Gamma)(1+L).
\end{equation}
\end{lemma}
\begin{proof}
For simplicity, we omit dependence of $H$. For every $\textbf{t} \in \R^d$, we have the following orthogonal splitting:
$$\Pi_{H,\alpha}^{-s}(\Gamma) = \B^{-s}_{\alpha,\textbf{t}}(\Gamma)B_{\alpha,\textbf{t}} + R_{\alpha,\textbf{t}}, $$
where
$$\quad \B^{-s}_{\alpha,\textbf{t}}:= \B^{-s}_{H,r_\textbf{-t}(\alpha)},  \ B_{\alpha,\textbf{t}}:=B_{r_\textbf{-t}(\alpha)}^{-s,H}, \ R_{\alpha,\textbf{t}}:=R_{r_\textbf{-t}(\alpha)}^{-s,H}.$$
For any $\textbf{h} \in \R^d$, we have
$$\B^{-s}_{\alpha,\textbf{t+h}}(\Gamma)B_{\alpha,\textbf{t+h}} + R_{\alpha,\textbf{t+h}} = \B^{-s}_{\alpha,\textbf{t}}(\Gamma)B_{\alpha,\textbf{t}} + R_{\alpha,\textbf{t}}.$$

By reparametrization \eqref{eqn;interntwine}, we have $B_{\textbf{t+h}} = e^{-(h_1+\cdots +h_d)/2}B_\textbf{t}$ and
\begin{equation}\label{eqn;rel}
\B^{-s}_{\alpha,\textbf{t+h}}(\Gamma) = e^{(h_1+\cdots +h_d)/2}\B^{-s}_{\alpha,\textbf{t}}(\Gamma) + \B^{-s}_{\alpha,\textbf{t+h}}(R_{\alpha,\textbf{t}})
\end{equation}
and it follows that
\begin{equation*}
\B^{-s}_{\alpha,\textbf{t+h}}(\Gamma) = e^{h_1/2}\B^{-s}_{\alpha,t_1, t_2+h_2,\cdots, t_d+h_d }(\Gamma) + \B^{-s}_{\alpha,\textbf{t+h}}(R_{\alpha,\textbf{t}}).
\end{equation*}

By differentiating at $h_1 =0$,
\begin{equation}\label{eq1}
\frac{d}{dt_1}\B^{-s}_{\alpha,t_1, t_2+h_2,\cdots, t_d+h_d}(\Gamma) = \frac{1}{2}\B^{-s}_{\alpha,t_1, t_2+h_2,\cdots, t_d+h_d}(\Gamma) + [\frac{d}{dh_1}\B^{-s}_{\alpha,\textbf{t+h}}(R_{\alpha,\textbf{t}})]_{h_1=0}.
\end{equation}

Therefore, we solve the following first order ODE
$$\frac{d}{dt_1}\B^{-s}_{\alpha,t_1, t_2+h_2,\cdots , t_d+h_d}(\Gamma) = \frac{1}{2}\B^{-s}_{\alpha,t_1, t_2+h_2,\cdots, t_d+h_d}(\Gamma) + \K^{(1)}_{\alpha,\bm{t},s}(\Gamma) $$
where
$$\K^{(1)}_{\alpha,\bm{t},s}(\Gamma) := [\frac{d}{dh_1}\B^{-s}_{\alpha,\textbf{t+h}}(R_{\alpha,\textbf{t}})]_{h_1=0}.$$

Then, the solution of the differential equation is
\begin{align*}
\B^{-s}_{\alpha,t_1, t_2+h_2,\cdots, t_d+h_d}(\Gamma) &= e^{t_1/2}\left(\B^{-s}_{\alpha,0, t_2+h_2,\cdots, t_d+h_d}(\Gamma)+\int_0^{t_1} e^{-\tau_1 /2}\K^{(1)}_{\alpha,\tau,s}(\Gamma)d\tau_1 \right)\\
&= e^{t_1/2}\B^{-s}_{\alpha,0, t_2+h_2,\cdots, t_d+h_d}(\Gamma)+\int_0^{t_1} e^{(t_1-\tau_1) /2}\K^{(1)}_{\alpha,\tau,s}(\Gamma)d\tau_1.
\end{align*}
Note by reparametrization
$$e^{t_1/2}\B^{-s}_{\alpha,0, t_2+h_2,\cdots, t_d+h_d}(\Gamma) = e^{h_2/2}\B^{-s}_{\alpha,t_1, t_2, t_3+h_3,\cdots, t_d+h_d}(\Gamma)$$
and it is possible to differentiate the previous equation at $h_2 = 0$ again. Then
\begin{equation}\label{eqn;ode2}
\frac{d}{dt_2}\B^{-s}_{\alpha,t_1, t_2,\cdots, t_d+h_d}(\Gamma) = \frac{1}{2}\B^{-s}_{\alpha,t_1, t_2, t_3+h_3, \cdots, t_d+h_d}+\int_0^{t_1} e^{(t_1-\tau_1) /2} \K^{(2)}_{\alpha,\tau,s}(\Gamma)d\tau_1
\end{equation}
where $\K^{(2)}_{\alpha,\tau,s}(\Gamma) = [\frac{d}{dh_2}\K^{(1)}_{\alpha,\tau,s}(\Gamma)]_{h_2 = 0}$.

Then, the solution of equation \eqref{eqn;ode2} is obtained by reparametrization
\begin{align*}
&\B^{-s}_{\alpha,t_1, t_2,\cdots, t_d+h_d}(\Gamma)   \\
&  = e^{t_2/2}\Big(\B^{-s}_{\alpha,t_1, 0,t_3+h_3,\cdots, t_d+h_d}+\int_0^{t_2} e^{-\tau_2 /2}\int_0^{t_1} e^{(t_1-\tau_1) /2} \K^{(2)}_{\alpha,\tau,s}(\Gamma)d\tau_1d\tau_2\Big)\\
&  =  e^{h_3/2}\B^{-s}_{\alpha,t_1, t_2,t_3,\cdots, t_d+h_d}+\int_0^{t_2} e^{(t_2-\tau_2) /2}\int_0^{t_1} e^{(t_1-\tau_1) /2} \K^{(2)}_{\alpha,\tau,s}(\Gamma)d\tau_1d\tau_2.
\end{align*}

Inductively, we solve the first order ODE repeatedly and obtain the following  solution
\begin{equation}\label{eqn;odefinal}
\B^{-s}_{\alpha,\bm{t}}(\Gamma) = e^{(t_1+\cdots + t_d)/2}\Big(\B^{-s}_{\alpha,0}+\int_0^{t_d}\cdots \int_0^{t_1} e^{-(\tau_1+\cdots \tau_d)/2}\K^{(d)}_{\alpha,\tau,s}(\Gamma)d\tau_1\cdots d\tau_d \Big)
\end{equation}
where
$$\K^{(d)}_{\alpha,\textbf{t},s}(\Gamma) = [\frac{d}{dh_d}\cdots \frac{d}{dh_1}\B^{-s}_{\alpha,\textbf{t+h}}(R_{\alpha,\textbf{t}})]_{h_d,\cdots, h_1=0}.$$

Let $\langle \cdot, \cdot \rangle_{\alpha, \bm{t}}$ denote the inner product in the space of Hilbert current $A_d(\mathfrak{p}, W_{r_\textbf{-t}(\alpha)}^{-s}\\(H))$. 
By the intertwining formula \eqref{eqn;interntwining},
\begin{align*}
\B^{-s}_{\alpha,\bm{t+h}}(R_{\alpha,\bm{t}}) & = \langle R_{\alpha,\bm{t}} , \frac{B_{\alpha,\bm{t+h}}}{|B_{\alpha,\bm{t+h}}|_{\bm{t+h}}^2}\rangle_{\alpha,\bm{t+h}}\\
& =  \langle R_{\alpha,\bm{t}}\circ U_{\bm{-h}} ,\frac{B_{\alpha,\bm{t+h}}\circ U_{\bm{-h}}}{|B_{\alpha,\bm{t+h}}|_{\bm{t+h}}^2} \rangle_{\alpha, \bm{t}}\\
& =  \langle R_{\alpha,\bm{t}}\circ U_{\bm{-h}} , \frac{B_{\alpha,\bm{t}}}{|B_{\alpha,\bm{t}}|_{\bm{t}}^2} \rangle_{\alpha, \bm{t}} = \B^{-s}_{\alpha,\bm{t}}(R_{\alpha,\bm{t}}\circ U_{\bm{-h}}).
\end{align*}

In the sense of distributions,
\begin{align*}
\frac{d}{dh_d}\cdots \frac{d}{dh_1}(R_{\alpha,\bm{t}}\circ U_{\bm{-h}}) &= -R_{\alpha,\bm{t}}\circ(\frac{d}{2} + \sum_{i=1}^{d}{X_i}(t))\circ U_{\bm{-h}} \\
& = [(\sum_{i=1}^{d}{X_i}(t) - \frac{d}{2})R_{\alpha,\bm{t}}]\circ U_{\bm{-h}}.
\end{align*}
Then we compute 
$$[\frac{d}{dh_d}\cdots \frac{d}{dh_1}(\B^{-s}_{\alpha,\bm{t+h}}(R_{\alpha,\bm{t}}))]_{\bm{h}=0} = -\B^{-s}_{\alpha,\bm{t}}((\sum_{i=1}^{d}{X_i}(t) - \frac{d}{2})R_{\alpha,\bm{t}}).$$

Recall by \eqref{def;kts} that
$\K_{\alpha,\textbf{t},s}(\Gamma) = \norm{\RR^{-s}_{\alpha,\bm{t}}}_{r_{-\bm{t}}(\alpha),-(s+1)}$. By Proposition \ref{prop;jackma} and Lemma \ref{04}, 
\begin{equation}\label{eqn;remainders}
 |\B^{-s}_{\alpha,\bm{t}}((\sum_{i=1}^{d}{X_i}(t) - \frac{d}{2})R_{\alpha,\bm{t}})| \leq  \K_{\alpha,\textbf{t},s}(\Gamma)
 \leq  C(s,\Gamma) Hgt ([[r_{-\bm{t}}(\alpha)]] )^{1/4}.
\end{equation}

Therefore, the solution of equation \eqref{eqn;odefinal} exists under Diophantine condition (\ref{cond;newdc}) and  the following holds:
$$\lim_{t_1, \dotsc, t_d \rightarrow \infty} e^{-(t_1+\cdots + t_d)/2}\B^{-s}_{\alpha,t}(\Gamma) = \hat\beta_H(\alpha,\Gamma).$$
Moreover, the complex number
$$\hat\beta_H(\alpha,\Gamma) = \B^{-s}_{\alpha,0}+\int_0^\infty\cdots \int_0^{\infty} e^{-(\tau_1+\cdots + \tau_d)/2}\K_{\alpha,\tau,s}(\Gamma)d\tau_1\cdots d\tau_d  $$
depends continuously on $\alpha \in DC(L)$. Since we have
$$\Pi_{H,\alpha}^{-s}(\Gamma) - \hat\beta(\alpha,\Gamma)B_\alpha^H = R_0 - \left(\int_0^\infty\cdots\int_0^\infty e^{-(\tau_1+ \cdots + \tau_d) /2}\K_{\alpha,\tau,s}(\Gamma)d\tau_1\cdots d\tau_d\right)B_\alpha^H,$$
by the Diophantine condition again,
$$\norm{\Pi_{H,\alpha}^{-s}(\Gamma) - \hat\beta_H(\alpha,\Gamma)B_\alpha^H}_{\alpha,-s} \leq  C(s,\Gamma)(1+L).$$
\end{proof}

\subsection{Proof of Theorem \ref{hi}}
The proof of Theorem \ref{hi} follows immediately from the refinement to the constructions of Bufetov functionals (see also \cite[\S 2.5]{BF14} for horocycle flows).

\emph{Notation.} The action of flow $\{r_t\}_{t \in \R}$ on a current $\mathcal{C}$ is defined by pull-back as follows:
\[
(r_t^*\mathcal{C})(\omega) = \mathcal{C}(r_{-t}^*\omega), \quad \text{for any smooth form } \omega.
\]

\begin{lemma}[Invariance]\label{lem;invar}
Let $1 \leq d \leq g$. The functional $\hat\beta_H$ defined on $d$-standard rectangle $\Gamma = \Gamma^X_\textbf{T}$ is invariant under the action of $(\QQ_y^{j,Y})$ for any $y \in \R_+^j$ and $1\leq j\leq d$. That is,
$$\hat\beta_H(\alpha,(\QQ^{j,Y}_y)_*\Gamma) = \hat\beta_H(\alpha,\Gamma).$$
\end{lemma}
\begin{proof}
We will prove the functional $\hat\beta_H$ on the rectangle $(\QQ^{j,Y}_y)_*\Gamma$ exists and prove its invariance property under the action $\QQ_y^{j,Y}$. It suffices to verify the invariance under the rank 1 action $\QQ^{1,Y}_\tau$ for $\tau \in \R$ since we can apply the statement for $d$-rank actions repeatedly.

Given a standard $d$-dimensional rectangle $\Gamma$, set $\Gamma_\QQ := (\QQ^{1,Y}_\tau)_*\Gamma$. Let $D(\Gamma, \Gamma_\QQ)$ be the $(d+1)$ dimensional space spanned by trajectories of the action of $\QQ^{1,Y}_\tau$ projecting $\Gamma$ onto $\Gamma_\QQ$.  $D(\Gamma, \Gamma_\QQ)$ is a union of all orbits I of action $\QQ_\tau^{1,Y}$ such that the boundary of $I$ ($d$-dimensional faces) is contained in $\Gamma \cup \Gamma_\QQ$. Then the interior of $I$ is disjoint from $\Gamma \cup \Gamma_\QQ$ and $D(\Gamma, \Gamma_\QQ)$ is defined by integration, i.e it is a $(d+1)$-current.

For convenience, let us denote renormalization flow by $r_t := r_{i}^{t}$ for some $i$-th coordinates. Then, $r_{-t}(\Gamma)$ and $r_{-t}(\Gamma_\QQ)$ are respectively the support of the currents $r_{t}^*\Gamma$ and $r_{t}^*\Gamma_\QQ$. Thus, we have the following identity
$$r_{t}^* D(\Gamma, \Gamma_\QQ) =  D({r_{-t}}(\Gamma), {r_{-t}}(\Gamma_\QQ)).$$
Since the current $\partial D(\Gamma, \Gamma_\QQ) - (\Gamma - \Gamma_\QQ)$ is composed of orbits for the action $\QQ^{1,Y}_\tau$,  it follows that
\begin{equation}\label{area0}
\partial[{r^*_t}D(\Gamma, \Gamma_\QQ) ] - ({r^*_t}\Gamma - {r^*_t}\Gamma_\QQ)= {r^*_t}[\partial D(\Gamma, \Gamma_\QQ) - (\Gamma - \Gamma_\QQ)] \rightarrow 0 \text{ as } t \rightarrow \infty.
\end{equation}

Now, we turn to prove the volume of $D({r_{-t}}(\Gamma), {r_{-t}}(\Gamma_\QQ))$ is uniformly bounded for all $t>0$. For any $p \in \Gamma$, set $\tau(p)$ be the length of arc on $ D_Q:= D(\Gamma, \Gamma_\QQ)$ and $\tau_\Gamma := \sup\{\tau(p) \mid p \in \Gamma\} < \infty$. We  write
$$vol_{d+1}( D_Q) = \int_{\Gamma} \tau dvol_d.$$
Since $vol_{d}(r_{-t}(\Gamma)) \leq e^tvol_{d}(\Gamma)$,
\begin{equation}\label{ineq;area2}
vol_{d+1}(r_{-t} (D_Q)) = \int_{r_{-t}(\Gamma)} \tau dvol_d \leq \tau_\Gamma e^{-t} vol_{d}(r_{-t}(\Gamma)) \leq \tau_\Gamma vol_{d}(\Gamma) < \infty.
\end{equation}

Note that $d$-dimensional current $(\QQ^{1,Y}_\tau)_*\Gamma - \Gamma$ is equal to the boundary of the $(d+1)$ dimensional current $D_Q$. By the same argument in remainder estimate (see Lemma \ref{3.1} and \eqref{ineq;sobolev}),
\begin{equation}\label{area1}
\norm{(\QQ^{1,Y}_\tau)_*\Gamma - \Gamma}_{r^{-t}(\alpha), -s} \leq C_s\tau B_s([[r_{-t} \alpha]])  \leq C_s\tau \text{Hgt} [[r_{-t}\alpha]]^{1/4}
\end{equation}
is finite for all $t >0$.

Then, by \eqref{area0}, \eqref{ineq;area2} and existence of Bufetov functional $\hat\beta_H(\alpha,\Gamma)$, the last inequality holds:
 $$ \norm{\B^{-s}_{\alpha,t}((\QQ^{1,Y}_\tau)_*\Gamma ) - \B^{-s}_{\alpha,t}(\Gamma )}_{\alpha,-s} < \infty.$$
Therefore, by the definition of Bufetov functional in the Lemma \ref{35}, $\hat\beta_H(\alpha,(\QQ^{1,Y}_\tau)_*\Gamma)$ exists and $\hat\beta_H(\alpha,\Gamma)$ is invariant under the action of $\QQ^{1,Y}_\tau$.
\end{proof}

\begin{figure}\label{fig1}\centering
\begin{tikzpicture}[scale=0.70]
 \begin{scope}[x={(4cm,0cm)},y={({cos(30)*1.5cm},{sin(30)*1.5cm})},
    z={({cos(70)*2cm},{sin(70)*2cm})},line join=round,fill opacity=0.55,thick]
  \draw[blue, dashed] (0,0,0) -- (0,0,1) --  (0,1,1)  -- (0,1,0) -- cycle  ;
  \draw[fill=gray] (0,0,0) -- (1,0,0) -- (1,1,0) -- (0,1,0) -- cycle;
   \draw[fill=orange] (0,0,1) -- (1,0,1) -- (1,1,1) -- (0,1,1) -- cycle;
  \draw[blue,dashed]  (1,1,1) -- (1,1,0) ;

  \draw[blue, dashed]  (1,0,0) -- (1,0,1) ;
    \draw[->]   (11/10,1,2/10) -- (11/10,1,8/10) node[right] {$\QQ_\tau^{d,Y}$} ; 
 \node at (1/2,1/2, 0)   {$\Gamma$};
\node at (1/2,1/2, 1)   {$\Gamma_Q$};
\node at (1/2,1/2, 1/2)   {$D(\Gamma, \Gamma_Q)$};
\node at (1,2/8, 3/5)   {$I$};

  \draw[fill=gray] (2,0,0) -- (3.5,0,0) -- (3.5,1.5,0) -- (2,1.5,0) -- cycle;
   \draw[fill=orange] (2,0,0.3) -- (3.5,0,0.3) -- (3.5,1.5,0.3) -- (2,1.5,0.3) -- cycle;
  \draw[blue, dashed] (2,0,0) -- (2,0,0.3) --  (2,1.5,0.3)  -- (2,1.5,0) -- cycle  ;

  \draw[blue,dashed]  (3.5,1.5,0.3) -- (3.5,1.5,0) ;
  \draw[blue, dashed]  (3.5,0,0) -- (3.5,0,0.3) ;
    \draw[red,dotted]  (2,0,0.3) -- (2.05,0,0) ;
        \draw[red,dotted]  (3.45,0,0.3) -- (3.5,0,0) ;
                \draw[red,dotted]  (3.45,1.5,0.3) -- (3.5,1.5,0) ;
 \node at (1/2,1/2, 0)   {$\Gamma$};
\node at (1/2,1/2, 1)   {$\Gamma_Q$};
\node at (1/2,1/2, 1/2)   {$D(\Gamma, \Gamma_Q)$};
\node at (1,2/8, 3/5)   {$I$};
    \draw[->]   (11/10,1,2/10) -- (11/10,1,8/10) node[right] {$\QQ_\tau^{d,Y}$} ; 

    \node at (2.8,1/2, 0.4)   {$r_{-t}(\Gamma_Q)$};
        \node at (2.8,1/2, 0.4)   {$r_{-t}(\Gamma_Q)$};
 \end{scope}

\end{tikzpicture}
    \caption{Illustration of the standard $d$-rectangles $\Gamma$, $\Gamma_Q$, $d+1$ dimensional current $D(\Gamma, \Gamma_\QQ)$ and supports of $r_{-t}(\Gamma)$ and $r_{-t}(\Gamma_\QQ) $.}
\end{figure}
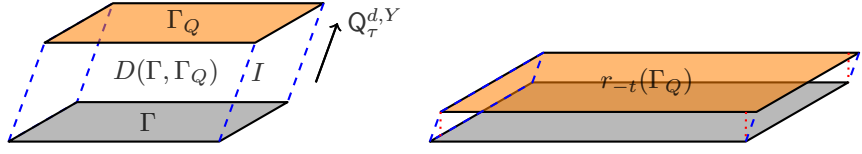

\begin{proof}[Proof of Theorem \ref{hi}]
\emph{Additive property.} It follows from the linearity of projections and limit.

\emph{Scaling property.} It is immediate from the definition.

\emph{Bounded property}. By the scaling property, for $r_t = r_{1}^{t}\cdots r_{d}^{t}$ with $t>0$,
$$\hat\beta_H(\alpha,\Gamma) = e^{dt/2}\hat\beta_H(r_t(\alpha),\Gamma).$$
Choose $t = \log (\int_{\Gamma}|\hat X| )$ and $\hat X = \hat X_1\wedge \cdots \wedge \hat X_d$, then Bufetov functional on the rectangle $\Gamma$ is bounded:
$$|\hat\beta_H(\alpha,\Gamma)| \leq C(\Gamma) (\int_{\Gamma}|\hat X| )^{d/2}.$$
\emph{Invariance property}. Now it follows directly from the Lemma \ref{lem;invar}.
\end{proof}

\noindent \emph{Notation.} We write for $\bm{T} = (T^{(i)}) \in \R^d$ and $\bm{t} = (t_1,\cdots, t_d) \in \R^d$,
\begin{equation}\label{not;vol}
vol(U(\bm{T})) := \prod_{i=1}^d{T^{(i)}}, \quad vol(U(\bm{t})) = \prod_{i=1}^d{t_i}.
\end{equation}

Now we extend the properties of functional $\hat\beta_H$ to cocycle $\beta_H$.
\begin{proof}[Proof of Corollary \ref{hi2}]
Cocycle property of $\beta_H$ follows from the additive property of $\hat\beta_H$. Scaling and bounded properties are immediate.

Denote $\bm{tT} := (t_1T^{(1)},\cdots, t_dT^{(d)}) \in \R_+^d$, and  we obtain
\begin{equation}\label{eqn;excursion tT}
\beta_H(\alpha,m,\bm{tT}) = vol(U(\bm{T}))^{1/2}\beta_H(r_{\log \textbf{T}}(\alpha),m,\textbf{t}).
\end{equation}

By Lemma \ref{35} and scaling property, we have
\begin{align*}
D_\alpha^H(f)\beta_{H}(\alpha,m,\textbf{t}) &= \lim_{|U(\bm{T})| \rightarrow \infty}\frac{1}{vol(U(\bm{T}))^{1/2}}\beta_{H}(r_{-\log\textbf{t}}(\alpha),m,\bm{tT})\\
& = \lim_{|U(\bm{T})| \rightarrow \infty}\frac{1}{vol(U(\bm{T}))^{1/2}}\left\langle \PPP^{d,r_{-\log\textbf{t}}(\alpha)}_{U(\bm{tT})}m , \omega_{f_H} \right\rangle.
\end{align*}
It follows that $\beta_H(\alpha,\cdot,\bm{t}) \in H$ as a point-wise limit of Birkhoff integrals.
This implies {orthogonal property}.
\end{proof}

For $\gamma \in Sp_{2g}(\Z)$, we have
\[
\beta_H(\gamma\alpha, \gamma(m),\bm{T}) = \beta_H(\alpha, m,\bm{T}).
\]
It means that the function $\beta_H(\cdot,m,\bm{T})$ is well-defined on the moduli space $\mathfrak{M}_g$.

\subsection{Proof of Theorem \ref{6.2type}}
We define the excursion function
\begin{align*}
&E_\M(\alpha,\bf{T})\\ & := \int_0^{\log T^{(d)}}\cdots \int_0^{\log T^{(1)}} e^{-(t_1 + \cdots + t_d)/2}\text{Hgt} ([[r_{\bf{t-\log T}}(\alpha)]])^{1/4} dt_1\cdots dt_d \\
& = {vol(U(\bm{T}))}^{1/2} \int_0^{\log T^{(d)}}\cdots \int_0^{\log T^{(1)}} e^{(t_1 + \cdots + t_d)/2}\text{Hgt} ([[r_{\bf{t}}(\alpha)]])^{1/4} dt_1\cdots dt_d.
\end{align*}

We prove that the cocycle $\beta^f$ in form of \eqref{BBB} is defined by a uniformly convergent series.
\begin{lemma}\label{6.1type} For any Diophantine  $[\alpha] \in DC(L)$,
$$|\beta^f(\alpha,m,\bm{tT})| \leq C_s\left(L + vol(U(\bm{T}))^{1/2}(1+vol(U(\bm{t}))+E_\M(\alpha,\bm{T}))\right)\norm{f}_{\alpha,s}$$
for any $f \in W^s(M)$ and $s > s_{d,g}+1/2$.
\end{lemma}
\begin{proof}
In view of Lemma \ref{35}, there exists a constant $C >0$ such that whenever $[\alpha] \in DC(L)$,
\begin{equation}\label{eqn;excursion}
|\beta_H(\alpha,m,\textbf{t})| \leq C(1+L+vol(U(\bm{t}))), \quad (m,\textbf{t}) \in M \times \R_+^d.
\end{equation}

By Diophantine condition (\ref{cond;newdc}),  $[r_{\log \textbf{T}}(\alpha)] \in DC(L_{\textbf T}) $ and
\begin{equation}\label{eqn;LT}
L_{\textbf T} \leq Lvol(U(\bm{T}))^{-1/2} + E_\M(\alpha,\bf{T}).
\end{equation}
Thus by \eqref{eqn;excursion}, for all $(m,\textbf{t}) \in M \times \R^d$, we obtain 
$$|\beta_H(r_{\log \textbf{T}}(\alpha),m,\textbf{t})| \leq C(1+L_{\textbf T} + vol(U(\bm{t}))).$$
By the scaling property \eqref{eqn;excursion tT}, it follows that for all $s > s_{d,g}$ and $q>1/2$
\begin{align*}
|\beta^f(\alpha,m,\bm{tT})| &\leq C_s vol(U(\bm{T}))^{1/2}(1+L+vol(U(\bm{t})))\sum_{n\in \Z}\norm{f_{n}}_{\alpha,s}  \\
& \leq C_s vol(U(\bm{T}))^{1/2}(1+L_{\textbf T}+vol(U(\bm{t})))\big(\sum_{n\in \Z}(1+n^2)^{q}\big)^{-1/2}\\
& \times (\sum_{n\in \Z}\norm{(1-Z^2)^{q/2}f_n}^2_{\alpha,s})^{1/2}.
\end{align*}
Therefore, for all $s' = s_{d,g} + q$, there exists a constant $C_{s'}>0$ such that
$$|\beta^f(\alpha,m,\textbf{tT})| \leq C_{s'}vol(U(\bm{T}))^{1/2}\big(1+L_{\textbf T}+vol(U(\bm{t}))\big)\norm{f}_{\alpha,s'}. $$
By combining it with \eqref{eqn;LT}, we obtain the statement.
\end{proof}

By Lemma \ref{35}, \ref{6.1type} and identification of the norm for the form $\omega_f$, asymptotic formula on each irreducible component provides the following corollary.
\begin{corollary}\label{cor;dev}
For all $s > s_{d,g}+1/2$, there exists a constant $C_s >0$ such that
for all $[\alpha] \in DC(L)$, for all $f \in W_\alpha^s(M)$  and for all $(m,\bm{T}) \in M \times \R^d$, we have
\begin{equation}\label{eqn:asymp}
|\left\langle \PPP^{d,\alpha}_{U(\bm{T})}m , \omega_f \right\rangle - \beta^f(\alpha,m,\bm{T}) | \leq C_s(1+L)\norm{\omega_f}_{\alpha,s}
\end{equation}
for $U(\bm{T}) = [0,T^{(1)}]\times \cdots \times [0,T^{(d)}]$ and $\omega_f = f\omega^{d,\alpha} \in  \Lambda^d\mathfrak{p} \otimes W_\alpha^s(M)$.
\end{corollary}

\begin{proof}[Proof of Theorem \ref{6.2type}]

The theorem follows from Corollary \ref{cor;dev} for $\alpha \in \bigcup_{L>0}\break DC(L)$. 
\end{proof}

\section{Limit distributions}\label{sec;4}
In this section, we prove Theorem \ref{limit}, limit distribution of normalized ergodic integrals of higher rank actions on the standard rectangles.

\subsection{Proof of Theorem \ref{limit}}

\begin{lemma}\label{7.1} There exists a continuous modular function $\theta_H : Aut_0(\mathsf{H}^g) \rightarrow H \subset L^2(M)$ such that for any $\omega_f = f\omega^{d,\alpha} \in  \Lambda^d\mathfrak{p} \otimes W_\alpha^s(H)$ with $s>d/2$,
\begin{equation}\label{eqn;7.1}
\lim_{|U(\bm{T})| \rightarrow \infty}\norm{\frac{1}{{vol(U(\bm{T}))}^{1/2}} \left\langle \PPP^{d,\alpha}_{U(\bm{T})}(\cdot) , \omega_f \right\rangle - \theta_H(r_{\log \bm{T}}(\alpha))D^H_\alpha(f)}_{L^2(M)} = 0.
\end{equation}
The family $\{ \theta_H(\alpha) \mid \alpha \in Aut_0(\mathsf{H}^g) \}$ has a constant norm in $L^2(M)$.
\end{lemma}

\begin{proof}
By the Fourier transform, the space of smooth vectors and Sobolev space $W^s(H)$ is represented as the Schwartz space $\mathscr{S}^s(\R^d) \subset L^2(\R^d)$ such that
$$\int_{\R^d}|(1+ \sum_{i=1}^d(\frac{\partial^2}{\partial u_i^2}+u_i^2)  )^{s/2} \hat f(u)|^2du < \infty. $$

Let $\bm{t, u} \in \R^d$. Then we claim for any $f \in \mathscr{S}^s(\R^d)$,  there exists a function $\theta(\alpha)(\cdot) \in L^2(\R^d)$ such that
$$\lim_{ |U(\bm{T})| \rightarrow \infty}\norm{\frac{1}{{vol(U(\bm{T}))}^{1/2}} \int_{U(\bm{T})} f(\bm{u+t})d\bm{t} - \theta_H(r_{\log \bm{T}}(\alpha))(\bm{u})Leb(f)}_{L^2(\R^d, d\bm{u})} = 0.  $$
This is equivalent to the statement \eqref{eqn;7.1}. By the standard Fourier transform on $\R^d$, equivalently
$$\lim_{|U(\bm{T})| \rightarrow \infty}\norm{\frac{1}{{vol(U(\bm{T}))}^{1/2}}  \int_{U(\bm{T})}e^{i \bm{t\cdot \hat u}}\hat f(\bm{\hat u}) dt - \hat\theta_H(r_{\log \bm{T}}(\alpha))(\bm{\hat u})\hat f(0) }_{L^2(\R^d, d\bm{\hat u})} = 0.     $$

For $\chi \in L^2(\R^d, d\bm{\hat u})$ and $\bm{\hat u} = (\hat u_j)_{1 \leq j \leq d}$, we denote
$$\chi_j(\bm{\hat u}) = \frac{e^{i \hat u_j} -1 }{i\hat u_j}, \ \chi(\bm{\hat u}) = \prod_{j=1}^d\chi_j(\bm{\hat u}).$$

Let $\hat\theta(\alpha)(\bm{\hat u}) := \chi(\bm{\hat u})$ for all $\bm{\hat u} \in \R^d$. Now we will compute $\theta(r_{\log \bm{T}} (\alpha))$.
By intertwining formula \eqref{eqn;interntwining} for $\textbf{T} \in \R^d$, 
$$U_\textbf{T}(f)(\bm{\hat u}) = {\prod_{i=1}^d (T^{(i)})^{1/2}}f(\bm{T\hat u}), \text{ for } \bm{T\hat u} = (T^{(1)}\hat u_1,\cdots, T^{(d)}\hat u_d).$$
Then, for all $\alpha \in Aut_0(\mathsf{H}^g)$,
$$\hat\theta(r_{\log \bm{T}} (\alpha))(\bm{\hat u}) = U_\textbf{T}(\chi)(\bm{\hat u}) = vol(U(\bm{T}))^{1/2}\chi(\bm{T\hat u}).$$
The function $\theta(\alpha)$ is defined by inverse Fourier transform of $\hat\theta(\alpha)$ and
$$\norm{\theta_H(\alpha)}_H = \norm{\theta(\alpha)}_{L^2(\R^d)} = \norm{\hat\theta(\alpha)}_{L^2(\R^d)} = \norm{\chi(\bm{\hat u})}_{{L^2(\R^d, d\hat u)}} = C >0.$$

By integration,
\begin{multline}
\int_{0}^{T^{(d)}}\cdots\int_{0}^{T^{(1)}}e^{i \bm{t\cdot \hat u}}\hat f(\bm{\hat u}) dt = vol(U(\bm{T}))\chi(\bm{T\hat u}) \hat f (\bm{\hat u})\\
 = vol(U(\bm{T}))\chi(\bm{T\hat u}) (\hat f (\bm{\hat u}) - \hat f (0)) + vol(U(\bm{T}))^{1/2}\hat\theta(r_{\log \bm{T}} (\alpha))(\bm{\hat u}) \hat f (0).
\end{multline}

Then the claim reduces to the following:
$$\limsup_{|U(\bm{T})| \rightarrow \infty} \norm{vol(U(\bm{T}))^{1/2}\chi(\bm{T\hat u})(\hat f(\bm{\hat u}) - \hat f( 0)) }_{L^2(\R^d)} = 0. $$
If $f \in \mathscr{S}^s(\R^g)$ with $s > d/2$, function $\hat f \in C^0(\R^d)$ and bounded. Thus, by Dominated convergence theorem and change of variables,
$$  \norm{vol(U(\bm{T}))^{1/2}\chi(\bm{T\hat u}) (\hat f (\bm{\hat u}) - \hat f (0))}_{L^2(\R^d, d\bm{\hat u})} =  \norm{\chi(\nu)(\hat{f}(\frac{\bm{\nu}}{\bm{T}}) - \hat f (0))  }_{L^2(\R^d,  d\bm{\nu})} \rightarrow 0 . $$
\end{proof}

\begin{corollary}\label{522} For any $s > d/2$, $\alpha \in Aut_0(\mathsf{H}^g)$ and $f \in W_\alpha^s(H)$, there exists a constant $C >0$ such that
$$\lim_{|U(\bm{T})|\rightarrow \infty} \frac{1}{{vol(U(\bm{T}))}^{1/2}}\norm{ \left\langle \PPP^{d,\alpha}_{U(\bm{T})}m , \omega_f \right\rangle }_{L^2(M)} = C|D_\alpha^H(f)|. $$
\end{corollary}

By Corollary \ref{522}, we derive the following limit result for the $L^2$-norm of Bufetov functionals.

\begin{corollary}\label{7.3} For every irreducible component $H$ and $[\alpha] \in DC$, there exists $C>0$ such that
$$\lim_{|U(\bm{T})|\rightarrow \infty} \frac{1}{{vol(U(\bm{T}))}^{1/2}}\norm{\beta_H(\alpha,\cdot,\bm{T})}_{L^2(M)} = C.$$
\end{corollary}
\begin{proof}
By the normalization of invariant distribution in Sobolev space, for any $\alpha \in Aut_0(\mathsf{H}^g)$,
there exists a function $f_\alpha^H \in W_\alpha^s(H)$ such that $D_\alpha(f_\alpha^H) = \norm{f_\alpha^H}_s = 1.$ For all $[\alpha] \in DC(L)$, 
by asymptotic formula \eqref{eqn:asymp} for $f = f_H$,
$$\big|\left\langle \PPP^{d,\alpha}_{U(\bm{T})}m , \omega_f \right\rangle - \beta^f(\alpha,m,\bm{T}) \big| \leq C_s(1+L).$$
Therefore, $L^2$-estimate follows from Corollary \ref{522}.
\end{proof}

A relation between the functional $\beta_H$ and the modular function $\theta_H$ is established below.

\begin{corollary}\label{4.4} For every irreducible component $H \subset L^2(M)$, the following holds.
For any $L >0$ and any $r_{\bf{t}}$-invariant probability measure $\mu$ supported on $DC(L) \subset \M_g$,
$$\beta_H(\alpha,\cdot, 1) = \theta_H(\alpha)(\cdot), \quad \text{for } \mu \text{-almost all } [\alpha] \in \M_g. $$
\end{corollary}
\begin{proof}
By Theorem \ref{6.2type} and Lemma \ref{7.1}, there exists a constant $C_\mu>0$ such that for all $[\alpha] \in supp(\mu) \subset DC(L)$ and $\textbf{T}\in \R_{+}^d$, we have
\begin{equation}
\lim_{|U(\bm{T})| \rightarrow \infty}\norm{ \beta_H(r_{\log \bm{T}}(\alpha),\cdot, 1) - \theta_H(r_{\log \bm{T}}(\alpha))}_{L^2(M)} \leq \frac{C_\mu}{{vol(U(\bm{T}))}^{1/2}}.
\end{equation}
By Luzin's theorem, for any $\delta >0$ there exists a compact subset $E(\delta) \subset \M $ such that we have the measure bound $\mu(\M \backslash E(\delta)) < \delta$ and the function $\beta_H(\alpha,\cdot,1) \in L^2(M)$ depends continuously on $[\alpha] \in E(\delta)$. By Poincaré recurrence, there is a full measure set $F \subset \M $ for $\R^d$-action. Denote by a full measure set $E'(\delta) = E(\delta) \cap F \subset E(\delta)$.

For every $\alpha_0 \in E'(\delta)$, there is a divergent sequence $(\bm{t}_n)$ such that $\{r_{\bm{t_n}}(\alpha_0) \} \subset E(\delta)$ and $\lim_{n \rightarrow \infty}r_{\bm{t}_n}(\alpha_0) = \alpha_0$. By continuity of $\theta_H$ and  $\beta_H$ at $\alpha_0$, we have
\begin{multline}\label{eqn;equality}
\norm{ \beta_H(\alpha_0,\cdot, 1) - \theta_H(\alpha_0)}_{L^2(M)} \\
= \lim_{n\rightarrow \infty}\norm{ \beta_H(r_{\bm{t}_n}(\alpha_0),\cdot, 1) - \theta_H(r_{\bm{t}_n}(\alpha_0))}_{L^2(M)} = 0.
\end{multline}
Then $\beta_H(\alpha,\cdot, 1) = \theta_H(\alpha) \in L^2(M)$ for all $\alpha \in E'(\delta)$. It follows that the set where the equality \eqref{eqn;equality} fails has a measure less than any $\delta>0$, thus the identity holds for $\mu$-almost all $\alpha \in Aut_0(\mathsf{H}^g)$.
\end{proof}

For all $\alpha \in Aut_0(\mathsf{H}^g)$, smooth function $f \in W^s(M)$ for $s > s_{d,g}+1/2$ decompose as an infinite sum, and the functional $\theta^f$ is defined by a convergent series
\begin{equation}
\theta^f(\alpha)  := \sum_H D_\alpha^H(f)\theta_{H}(\alpha).
\end{equation}
Hence, the modular function $\theta^f : Aut_0(\mathsf{H}^g) \rightarrow L^2(M)$ is continuous.

The following result is an extension of Lemma \ref{7.1} to an asymptotic formula.

\begin{lemma}\label{lll}
For all $\alpha \in Aut_0(\mathsf{H}^g)$, $f \in W^s(M)$ and $s > s_{d,g}+1/2$,
$$\lim_{n\rightarrow \infty}\norm{\frac{1}{vol(U(\bm{T_n}))^{1/2}}\left\langle \PPP^{d,\alpha}_{U(\bm{T_n})}m , \omega_f \right\rangle- \theta^f(r_{\log \bm{T_n}}(\alpha))}_{L^2(M)} = 0.$$
\end{lemma}

We summarize our results on limit distributions for higher rank actions.

\begin{theorem}[Theorem \ref{limit}]\label{thm;limit2}
Let $(\bm{T_n})$ be any sequence such that
$$\lim_{n\rightarrow \infty} r_{\log \bm{T_n}} [\alpha] = \alpha_\infty \in \M_g.$$
For every closed form $\omega_f \in \Lambda^d\mathfrak{p} \otimes W^s(M)$ with $s> s_{d,g}+1/2$, which is not a coboundary, the limit distribution of the family of random variables $E_{\bm{T_n}}(f)$
exists along a subsequence of $\{\bm{T_n}\}$ and is equal to the distribution of the function $\theta^f(\alpha_\infty) = \beta(\alpha_\infty, \cdot, 1) \in L^2(M)$.

If $\alpha_\infty \in DC$, then $\theta^f(\alpha_\infty)$ is a bounded function on $M$, and the limit distribution has compact support.
\end{theorem}

\begin{proof}[Proof of Theorem \ref{limit}] 
Since $\alpha_\infty \in \M_g$, the existence of limit follows from the Lemma \ref{lll} and Theorem \ref{thm;limit2}.
\end{proof}

A relation with Birkhoff integrals and theta sum was introduced in \cite[\S 5.3]{CF15}, and as an application, we derive the limit theorem of theta sums.

\begin{corollary} Let $\mathscr{Q}[x] = x^\top \mathscr{Q}x$ be the quadratic forms defined by $g \times g$ real matrix $\mathscr{Q}$, where $\alpha = \begin{pmatrix}{I}&{0}\\{\mathscr{Q}}&{I}\end{pmatrix} \in Sp_{2g}(\R)$ and $\ell(x) = \ell^\top x$ is the linear form defined by $\ell \in \R^g$. Then the theta sum
$$\Theta(\mathscr{Q},\ell;N)= N^{-g/2}\sum_{n \in \Z^g \cap [0,N]} \exp (2\pi \iota(\mathscr{Q}[n] + \ell(n))) $$
 has a limit distribution and it has compact support.
\end{corollary}

\section{$L^2$-lower bounds }\label{sec;5}
In this section we prove $L^2$- lower bounds of ergodic integrals on transverse torus.

\subsection{Structure of return map }\label{sec;returnmap}

The polarized Heisenberg group $\mathsf{H}^g_{pol} \approx \R^g \times \R^g \times \R $ is equipped with the group law $(x,y,z)\cdot (x',y',z') = (x+x',y+y',z+z'+yx')$.
Reduced standard Heisenberg group is defined by quotient
$$\mathsf{H}_{red}^g := \mathsf{H}_{pol}^g/(\{0\}\times \{0\} \times \frac{1}{2}\Z) \approx \R^g \times \R^g \times \R / \frac{1}{2}\Z.$$
Then the Reduced standard lattice is $\mathsf{\Gamma}_{red}^g = \Z^g\times \Z^g \times \{0\} \subset \mathsf{H}_{red}^g$ and the quotient $\mathsf{H}_{red}^g/\mathsf{\Gamma}_{red}^g$ is isomorphic to the standard Heisenberg manifold $M = \mathsf{H}^g/\mathsf{\Gamma}$.

Given the standard frame $(X_i,Y_i,Z)$,  $(g+1)$-dimensional (transverse) torus is denoted by
$$\T_\mathsf{\Gamma}^{g+1} := \{\mathsf{\Gamma}\exp(\sum_{i=1}^g y_iY_i +zZ) \mid ({y_i},z) \in \R\times \R \}.$$

Now, we consider a return map of $\PP^{d,\alpha}$ on $\T_\mathsf{\Gamma}^{g+1}$ on the coordinates in reduced Heisenberg group.
 For $x = (x_1, \cdots, x_g) \in \R^g$, we write in the coordinate of $\mathsf{H}_{red}^g$ for convenience
$$\exp(x_1X_1^\alpha + \cdots +x_gX_g^\alpha) = (x_\alpha, x_\beta, w\cdot x ), \text{ for some } x_\alpha, x_\beta, \text{ and } w \in \R^d.$$
Then by group law,
$$\exp(x_1X_1^\alpha + \cdots +x_gX_g^\alpha)\cdot(0,y,z) = (x_\alpha, y + x_\beta, z+ w\cdot x )$$
and given $(n,m,0) \in \mathsf{\Gamma}_{red}^g$,
\begin{equation}\label{def;return}
\exp(x_1X_1^\alpha + \cdots +x_gX_g^\alpha)\cdot(0,y,z)\cdot (n,m,0) = \exp(x'_1X_1^\alpha + \cdots +x'_gX_g^\alpha)\cdot (0,y',z' )
\end{equation}
if and only if
\[
x_\alpha' = x_\alpha+n, y' = y + (x_\beta - x_\beta')  +m \text{ and } z' = z+ (w - w')\cdot x + n^\top(y+x_\beta).
\]

Assume $\langle X^\alpha_i, X_j\rangle \neq 0$ for all $i, j$, and we write the first return time for $\PP^{d,\alpha}$ action
$$t_{Ret} = (t_{Ret,1},\cdots, t_{Ret,d}) \in \R^d$$
on transverse torus $\T_\mathsf{\Gamma}^{g+1}$. 
We denote the domain for return time $U(t_{Ret}) = [0,t_{Ret,1}]\\ \times \cdots \times [0,t_{Ret,d}]$. Return map of action $\PP^{d,\alpha}$ on $\T_\mathsf{\Gamma}^{g+1}$ has a form of skew-shift
\begin{equation}
A_{\rho,\tau}(y, z) = (y+\rho,  z+v\cdot y+\tau) \text{ on } \R^g/\Z^g \times \R/K^{-1}\Z
\end{equation}
for some non-zero vectors $\rho, v \in \R^g$ and $\tau \in \R$.

Furthermore, $A_{\rho,\tau} = A_{d,\rho,\tau} \circ \cdots \circ A_{1,\rho,\tau}$ decomposes with  commuting linear skew-shifts
\begin{equation}\label{skew}
A_{i,\rho,\tau}(y, z) = (y+\rho_i,  z+v_{i}\cdot y+\tau_i) \text{ on } \R^g/\Z^g \times \R/K^{-1}\Z
\end{equation}
for some $\rho_i, v_i \in \R^g$ and $v_{i} \in \R^g$. For each $j \neq k$, it is easily verified that
$$A_{j,\rho,\tau}\circ A_{k,\rho,\tau} = A_{k,\rho,\tau}\circ A_{j,\rho,\tau}.$$

Given pair $(\m,n) \in \Z_{K|n|}^{g}\times \Z$, let $H_{(\m,n)}$ denote the corresponding factor and $C^\infty(H_{(\m,n)})$ be a subspace of smooth functions on $H_{(\m,n)}$.
Denote $\{e_{\m,n}\mid (\m,n) \in \Z_{|n|}^{g}\times \Z \}$  the basis of characters on $\T_\mathsf{\Gamma}^{g+1}$ and for all $(y,z) \in \T^g \times \T$,
$$e_{\m,n}(y,z):= \exp[2\pi \iota (\m\cdot y + nKz)].$$
%
For each $A_{i,\rho,\sigma}$, we set $\bm{v_i} = (v_{i1},\cdots, v_{ig}) \in \Z_{K|n|}^{g}$. Then the orbit is identified  with the following dual orbit
\begin{align*}
\OO_{A_i}(\bm{m},n)  & = \{(\bm{m} + (nj_i)\bm{v_i},n), \ j_i \in  \Z \}\\
&= \{ (m_1 + (nv_{i1})j_i, \cdots, m_g + (nv_{ig})j_i, n ), \ j_i \in  \Z\} .
\end{align*}

If $n = 0$, the orbit $[(\m,0)] \subset \Z^{g}\times \Z$ of $(\m,0)$ is reduced to a single element. If $n \neq 0$,  then the dual orbit $[(\m, n)] \subset \Z^{g+1}$ of $(\m, n)$ for higher rank actions is described as follows:
$$\OO_{A}(\bm{m},n) = \{(m_k + n\sum_{i=1}^d(v_{ik} j_i), n)_{1\leq k \leq d} : j = (j_1, \cdots j_d)\in \Z^d\}.$$

It follows that every $A$-orbit for rank $\R^d$-action (or $A_i$-orbit) can be labeled uniquely by a pair $(\m, n) \in \Z_{|n|}^g \times \Z\backslash\{0\}$ with $\m = (m_1,\cdots, m_g)$.  Thus, the subspace of functions with non-zero central characters splits as a direct sum of components $H_{(\m,n)}$
$$L^2(\T_\mathsf{\Gamma}^{g+1}) = \bigoplus_{\omega \in \OO_A}H_\omega, \quad \text{where } H_\omega = \bigoplus_{(\m,n) \in \omega} \C e_{(\m,n)}.$$

\subsection{Higher cohomology for $\Z^d$-action of skew-shifts}
In this subsection, we find relations with the return map for $\Z^d$ action of $\PP^{d,\alpha}$ on the torus $\T_\mathsf{\Gamma}^{g+1}$ and obstructions for solving cohomological equation $\omega = d \Omega$.

We will restrict  our interest to the following cocycle equation
\begin{equation}\label{eqn;cohskew}
\varphi(x,t) =  \mathfrak D\Phi(x,t), \quad x \in \T^g, \ t \in \Z^d
\end{equation}
where $d$-cocycle $\varphi : \T_\mathsf{\Gamma}^{g+1} \times \Z^d \rightarrow \R$, $(d-1)$ cochain $\Phi : \T_\mathsf{\Gamma}^{g+1} \rightarrow \R^d$, $\Phi = (\Phi_1, \cdots, \Phi_d)$, and
$\mathfrak D$ is coboundary operator
$$\mathfrak D \Phi = \sum_{i=1}^{d}(-1)^{i+1}\Delta_{i} \Phi_i $$ where $\Delta_{i} \Phi_i = \Phi_i \circ A_{i,\rho,\tau} - \Phi_i$.

In the work of Katoks, they proved the existence of solutions for cocycle equations by studying the dual equations $\hat\varphi =  \mathfrak D \hat\Phi$ in the space of Fourier coefficients (dual orbit) \cite[\S 2]{katok1995higher}. We apply this result to our $A$-orbit and verify the invariant distributions (or currents) explicitly.

\begin{proposition}\cite[Proposition 2.2]{katok1995higher}
A dual cocycle $\hat\varphi$ satisfies cocycle equation \eqref{eqn;cohskew} if and only if $\sum_{j \in \Z^d}\hat\varphi_{(\m,n)} \circ A^j = 0$.
\end{proposition}
For fixed $(\m,n) \in \Z^g \times \Z$ , we denote an obstruction for solving cohomological equation restricted to the $A$-orbit of $(\m,n)$ by $\mathcal{D}_{\m,n}(\varphi) = \sum_{j \in \Z^d} \hat\varphi_{(\m,n)} \circ A^j$. Since $A^j$ is composition of commuting toral automorphisms,
we obtain the following generalized formula (see \cite[\S5]{AFU11} and \cite[\S 11]{Ka03} for rank 1 map on $\T^2$).
\begin{lemma}\label{skew} There exist distributional obstructions to the existence of a smooth solution $\varphi \in C^\infty(H_{(\m,n)})$ of the cohomological equation (\ref{eqn;cohskew}).
A generator of the space of invariant distribution $\D_{\m,n}$ is given by
$$\mathcal{D}_{\m,n}(e_{a,b}) := e^{-2\pi \iota \sum_{i=1}^d[(\m\cdot \rho_i + nK\tau_i ) j_i + nK\tau_i{j_i\choose 2}]}$$
if $(a,b) = (m_k+K\sum_{i=1}^d(v_{ik} j_i),n)_{1\leq k \leq g}$
and 0 otherwise.
\end{lemma}

\begin{proof}
From previous observation, there exists an obstruction
\begin{align}
\mathcal{D}_{\m,n}(\varphi) &= \sum_{j \in \Z^d}\int_{\T_\mathsf{\Gamma}^{g+1}}\varphi(x,y)\overline {e_{\m,n}\circ A_{\rho,\tau}^j}dxdy.
\end{align}
By direct computation, for fixed $j = (j_1, \cdots j_d)$,
$$e_{\m,n}\circ A_{\rho,\tau}^j(y,z) = \prod_{i=1}^d\big(e^{2\pi \iota [(\m\cdot \rho_i + nK\tau_i ) j_i + nK\tau_i{j_i\choose 2}]}\big)\big(e^{2\pi \iota (\m\cdot y+ K(z+n\sum_{k=1}^d(v_{ik} j_i)y_k))}\big).$$
Then,  we choose $\hat \varphi = e_{a,b}$ for $(a,b) = (m_k+K\sum_{i=1}^d(v_{ik} j_i),n)_{1\leq k \leq g}$ in the non-trivial orbit $(n\neq 0)$,
\begin{equation}\label{eqn;identity}
\mathcal{D}_{\m,n}(e_{a,b}) = e^{-2\pi \iota \sum_{i=1}^d[(\m\cdot \rho_i + nK\tau_i ) j_i + nK\tau_i{j_i\choose 2}]}.
\end{equation}
\end{proof}

\subsection{Changes of coordinates}
For any frame $(X_i^\alpha,Y_i^\alpha,Z)_{i=1}^g$ and any $m \in M$, denote a transverse cylinder
$$\mathscr{C}_{\alpha,m} := \{m\exp(\sum_{i=1}^g y_i'Y_i^\alpha+z'Z) \mid (y',z') \in  U(t_{Ret}^{-1})  \times \T\} \subset M.$$

For any $\xi \in \T_{\mathsf{\Gamma}}^{g+1}$, let $\xi' \in \mathscr{C}_{\alpha,m}$ denote first intersection of the orbit $\{\mathsf{P}_t^{d,\alpha}(\xi) \mid t \in \R_+^d  \}$ with transverse cylinder $\mathscr{C}_{\alpha,m}$.
Then, there exists a first return time to the cylinder $t(\xi) = (t_1(\xi), \cdots, t_d(\xi)) \in \R_+^d$ such that the map $\Phi_{\alpha,m} : \T_{\mathsf{\Gamma}}^{g+1} \rightarrow \mathscr{C}_{\alpha,m} $ is defined by
$$\xi' = \Phi_{\alpha,m}(\xi) = \mathsf{P}_{t(\xi)}^{d,\alpha}(\xi), \ \ \forall \xi \in \T_{\mathsf{\Gamma}}^{g+1}. $$

Let $(y,z)$ and $(y',z')$ denote the coordinates on $\T_{\mathsf{\Gamma}}^{g+1}$ and $\mathscr{C}_{\alpha,m}$ respectively, given by the exponential map
$$(y,z) \rightarrow \xi_{y,z}:= \mathsf{\Gamma}\exp(\sum_{i=1}^g y_iY_i+zZ), \quad (y',z') \rightarrow m\exp(\sum_{i=1}^g y_i'Y_i^\alpha+z'Z).$$

Recall that if $\alpha \in Sp_{2g}(\R)$, then for $1\leq i, j \leq g$ there exist matrices $A = (a_{ij}), B = (b_{ij}), C = (c_{ij}), D = (d_{ij})$ such that
\[
	\alpha :=
	\left[
	\begin{array}{cc}
	A &  B  \\
	C & D   \\
	\end{array}
	\right] \in Sp_{2g}(\R),
\]
satisfying $A^tD-C^tB = I_{2g}$, $C^tA = A^tC$, $D^tB = B^tD$, and $\det(A) \neq 0$.
Set
$$X_i^\alpha = \sum_{j = 1}^g(a_{ij}X_j + b_{ij}Y_j) + w_iZ, \quad Y_i^\alpha = \sum_{j = 1}^g(c_{ij}X_j + d_{ij}Y_j) + v_iZ.$$

 Let $x = \mathsf{\Gamma}\exp(\sum_{i=1}^d y_{x,i}Y_i+z_xZ)\exp(\sum_{i=1}^d t_{x,i}X_i) $, for some $(y_x,z_x) \in \T^d \times \R/K\Z$ and $t_x  = (t_{x,i})\in [0,1)^d$.
Then,  the map $\Phi_{\alpha,x} : \T_{\mathsf{\Gamma}}^{g+1} \rightarrow \mathscr{C}_{\alpha,x}$ is defined by
$\Phi_{\alpha,x}(y,z) = (y',z')$ where
\begin{equation}\label{11}
\begin{bmatrix}
    y'_{1}    \\
    y'_{2}    \\
    \vdots\\
    y'_{g}
\end{bmatrix} =
\begin{bmatrix}
    a_{11}       & a_{12} &  \cdots & a_{1g} \\
    a_{21}       & a_{22} & \cdots & a_{2g} \\
        \vdots&     \vdots&     \vdots& \vdots \\
    a_{g1}       & a_{g2} & \cdots & a_{gg}
\end{bmatrix}
\begin{bmatrix}
    y_{1} - y_{x,1}   \\
    y_{2} - y_{x,2}    \\
    \vdots\\
    y_{g} - y_{x,g}
\end{bmatrix} +
\begin{bmatrix}
    b_{11}       &   \cdots & b_{1g} \\
    b_{21}       &  \cdots & b_{2g} \\
        \vdots&          \vdots& \vdots \\
    b_{g1}       &  \cdots & b_{gg}
\end{bmatrix}
\begin{bmatrix}
    t_{x,1}   \\
    t_{x,2}    \\
    \vdots\\
    t_{x,g}
\end{bmatrix},
\end{equation}
and $z' = z + P(\alpha,x,y)$ for some degree 4 polynomial $P$.

Therefore, the map $\Phi_{\alpha,x}$ is invertible with
$$\Phi^*_{\alpha,x}(dy_1'\wedge \cdots dy_g' \wedge dz') =  \frac{1}{det(A)} dy_1\wedge \cdots dy_g \wedge dz.$$

Since $A^tD-C^tB = I_{2g}$, by direct computation we obtain a vector of return time 
\begin{equation}
\begin{bmatrix}
    t_1(\xi)    \\
    t_2(\xi)    \\
    \vdots\\
    t_g(\xi)
\end{bmatrix} =
\begin{bmatrix}
    d_{11}       &   \cdots & d_{1g} \\
    d_{21}       &  \cdots & d_{2g} \\
        \vdots&          \vdots& \vdots \\
    d_{d1}       &  \cdots & d_{gg}
\end{bmatrix}
\begin{bmatrix}
    t_{x,1}   \\
    t_{x,2}    \\
    \vdots\\
    t_{x,d}
\end{bmatrix}
+
\begin{bmatrix}
    c_{11}       & c_{12} &  \cdots & c_{1g} \\
    c_{21}       & c_{22} & \cdots & c_{2g} \\
        \vdots&     \vdots&     \vdots& \vdots \\
    c_{g1}       & c_{g2} & \cdots & c_{gg}
\end{bmatrix}
\begin{bmatrix}
    y_{1} - y_{x,1}   \\
    y_{2} - y_{x,2}    \\
    \vdots\\
    y_{g} - y_{x,g}
\end{bmatrix}.
\end{equation}
Then,
\begin{align*}
\norm{t(\xi)} \leq \max_i| t_i(\xi)|^g \leq
\max_i|\sum_{j=1}^g d_{ij} t_{x,i}+ c_{ij}(y_i -  y_{x,i}) |^g
 \leq \max_i\norm{Y^\alpha_i}^g.
 \end{align*}

\subsection{$L^2$-lower bound of functional}
We will prove the bounds for square mean of integrals along leaves of foliations of the torus $\T_\mathsf{\Gamma}^{g+1}$.

\begin{lemma}\label{8.1}
For all $\alpha = (X^\alpha_i,Y^\alpha_i,Z)$ and  every irreducible component $H:=H_n$ of central parameter $n\neq 0$, there exist a function $f_H$ and a constant $C>0$ such that 
$$|f_H|_{L^\infty(H)} \leq  Cvol{(U(t_{Ret}))}^{-1} |\mathcal{D}_\alpha^H(f_H)|,$$
$$\norm{f_H}_{\alpha,s} \leq Cvol{(U(t_{Ret}))}^{-1}|\mathcal{D}_\alpha^H(f_H)|\Big(1+ \frac{\Sigma(t_{Ret})}{vol{(U(t_{Ret}))}}\norm{Y}\Big)^s(1+n^2)^{s/2} $$
where $\norm{Y} := \max_{1\leq i \leq g}\norm{Y_i^\alpha}$ and $\Sigma(t_{Ret}) = \sum_{i=1}^g t_{Ret,i}$.\\

On rectangular domain ${U(\bm{T})}$, for all $m \in \T_\Gamma^{g+1}$ and $T^{(i)} \in \Z_{t_{Ret,i}}$
\begin{equation}\label{mind}
\norm{\left\langle \PPP^{d,\alpha}_{U(\bm{T})} (\QQ^{g,Y}_ym) , \omega_H \right\rangle}_{L^2(\T^g,dy)} = |\mathcal{D}_\alpha^H(f_H)|\left(\frac{vol{(U(\bm{T}))}}{vol{(U(t_{Ret}))}} \right)^{1/2}.
\end{equation}
In addition, whenever $H \perp H' \subset L^2(M)$ the functions
$$\left\langle \PPP^{d,\alpha}_{U(\bm{T})} (\QQ^{g,Y}_ym) , \omega_H \right\rangle \quad  \text{and} \quad \left\langle \PPP^{d,\alpha}_{U(\bm{T})} (\QQ^{g,Y}_ym) , \omega_{H'} \right\rangle $$
are orthogonal in $L^2(\T^g,dy)$.
\end{lemma}
\begin{proof}

The operator $I_\alpha : L^2(M) \rightarrow L^2(\T_\Gamma^{g+1})$ is defined by
\begin{equation}\label{changesof}
f \rightarrow I_\alpha(f) := \int_{U(t_{Ret})} f\circ \PP_s^{d,\alpha} (\cdot)ds.
\end{equation}
Then operator $I_\alpha$ is surjective linear map of $L^2(M)$ onto $L^2(\T_\Gamma^{g+1})$ with a right inverse $R_\alpha^{\chi}$ defined as follows. Let $\chi \in C_0^\infty (0,1)^g$ be any function of jointly integrable with integral 1. For any $F \in L^2(\T_\Gamma^{g+1})$, let $R_{\alpha}^{\chi}(F) \in L^2(M)$ be a function defined by
$$R_\alpha^{\chi}(F)(\mathsf{P}^{d,\alpha}_v(x)) =  \frac{1}{vol{(U(t_{Ret}))}} \chi(\frac{v}{t_{Ret}})F(x), \ (x,v) \in \T_\Gamma^{g+1} \times U(t_{Ret}). $$
 Then, it follows that there exists a constant $C_\chi>0$  such that
\begin{align}\label{eqn;Ralpha}
\begin{split}
\norm{R_\alpha^{\chi}(F)}_{\alpha,s} &\leq C_\chi vol{(U(t_{Ret}))}^{-1}(1+ \sum_{i=1}^g t^{-1}_{Ret,i}\norm{Y^\alpha_i})^s\norm{F}_{W^s(\T_\Gamma^{g+1})}\\
& \leq C_\chi vol{(U(t_{Ret}))}^{-1}\Big(1+ \frac{\Sigma(t_{Ret})}{vol{(U(t_{Ret}))}}\norm{Y}\Big)^s\norm{F}_{W^s(\T_\Gamma^{g+1})}.
\end{split}
\end{align}

As explained in \S \ref{sec;returnmap}, the space $L^2(\T_\Gamma^{g+1})$ decomposes as a direct sum of irreducible subspaces invariant under the action of each $A_{j,\rho,\sigma}$.
It follows that the subspace of functions with a non-zero central character can be split as direct sum of components $H_{(\m, n)}$ with $(\m, n) \in \Z_{|n|}^g \times \Z\backslash\{0\}$ with $\m = (m_1,\cdots, m_g)$. For a function $F \in H_{(\m, n)}$,  it is characterized by Fourier expansion
$$F = \sum_{j \in \Z^d} F_j e_{A^j(\bm{m},n)} = \sum_{j \in \Z^d} F_j e_{(m_k+K\sum_{i=1}^d(v_{ik} j_i),n)}.$$

Choose $f_H := R_\alpha^{\chi}(e_{\m,n}) \in C^\infty(H)$ such that 
\begin{align}
|D_\alpha(f^H_\alpha)| &= |D_{\m,n}(e_{\m,n})|= 1,\label{def:normalization}\\
\int_{U(t_{Ret})} f_H\circ \PP_t^{d,\alpha} (y,z)dt &= e_{\m,n}(y,z), \text{ for } (y,z) \in \T_\Gamma^{g+1}. \label{eqn;emn}
\end{align}


Therefore, it follows from \eqref{eqn;Ralpha} that
$$|f_H|_{L^\infty(H)} \leq  C_\chi vol{(U(t_{Ret}))}^{-1}, $$
$$\norm{f_H}_{\alpha,s} \leq Cvol{(U(t_{Ret}))}^{-1}|\D_\alpha^H(f_H)|\Big(1+ \frac{\Sigma(t_{Ret})}{vol{(U(t_{Ret}))}}\norm{Y}\Big)^s(1+n^2)^{s/2}. $$
Moreover, since $\{ e_{\m,n}\circ A^j_{\rho,\tau}  \}_{j \in \Z^d} \subset L^2(\T_\Gamma^{g},dy )$ is orthonormal, we  verify
\begin{align*}
\norm{\left\langle \PPP^{d,\alpha}_{U(\bm{T})} (\QQ^{g,Y}_yx) , \omega_H \right\rangle }_{L^2(\T^g,dy)} & = \norm{\sum_{j_d = 0}^{[\frac{T^{(d)}}{t_{Ret,d}}]-1}\cdots\sum_{j_1 = 0}^{[\frac{T^{(1)}}{t_{Ret,1}}]-1} e_{\m,n}\circ A^j_{\rho,\tau}}_{L^2(\T^g,dy)}\\
& = \left(\frac{vol{(U(\bm{T}))}}{vol{(U(t_{Ret}))}} \right)^{1/2}.
\end{align*}
\end{proof}

Recall that
\begin{equation}\label{def:irred}
L^2(M) = \bigoplus_{n\in \Z}H_n := \bigoplus_{n\in \Z}\bigoplus_{i=1}^{\mu(n)}H_{i,n}
\end{equation}
where $H_n =  \bigoplus_{i=1}^{\mu(n)}H_{i,n}$ is irreducible representation with a central parameter $n$ and $\mu(n)$ is countable by Howe-Richardson multiplicity formula.
\medskip

 For any infinite dimensional vector $\textbf{c} := (c_{i,n}) \in \ell^2$, let $\beta_\cc$ denote a weighted Bufetov functional
$$\beta_\cc = \sum_{n \in \Z}\sum_{i=1}^{\mu(n)} c_{i,n}\beta^{i,n}. $$
By orthogonal property and Corollary \ref{7.3}, the function $\beta_\textbf{c}(\alpha,\cdot,\textbf{T}) \in L^2(M)$ for all $(\alpha, \textbf{T}) \in Aut_0(\mathsf{H}^g) \times \R_+^d$. Furthermore,
\begin{equation*}
\norm{\beta_\textbf{c}(\alpha,\cdot,\textbf{T})}^2_{L^2(M)} = \sum_{n \in \Z}\sum_{i=1}^{\mu(n)} |c_{i,n}|^2 \norm{\beta^{i,n}(\alpha,\cdot,\textbf{T})}^2_{L^2(H_{i,n})}
 \leq C^2|\textbf{c}|_{\ell^2}^2vol{(U(\bm{T}))}.
\end{equation*}
For any $\textbf{c} := (c_{i,n})$, let $|\textbf{c}|_s $ denote the norm defined by
\begin{equation}\label{eqn;cnorm}
|\textbf{c}|_s^2 = \sum_{n \in \Z\backslash\{0\} }\sum_{i=1}^{\mu(n)}(1+K^2n^2)^s|c_{i,n}|^2.
\end{equation}

\begin{lemma}\label{Crucial}
For any $s > s_{d,g}+1/2$, there exists a constant $C_s >0$ such that for all $\alpha \in DC(L)$, for all $\cc \in \ell^2$, and for all $z \in \T$
\begin{multline*}
\left|\norm{\beta_{\cc}(\alpha,\Phi_{\alpha,x}(\xi_{y,z}),\bm{T})}_{L^2(\T^g,dy)} - \left(\frac{vol{(U(\bm{T}))}}{vol{(U(t_{Ret}))}} \right)^{1/2}|\cc|_0 \right| \\
\leq C_s (vol{(U(t_{Ret}))} + vol{(U(t_{Ret}))}^{-1}) (1+L)\Big(1+ \frac{\Sigma(t_{Ret})}{vol{(U(t_{Ret}))}}\norm{Y}\Big)^s|\textbf{c}|_{s}.
\end{multline*}
\end{lemma}
\begin{proof}
By Lemma \ref{skew}, there exists a function $f_{i,n} \in C^\infty(H_{i,n})$ with $|\D^{i,n}(f_{i,n})| = 1$. Let $f_{\cc} =\sum_{n \in \Z}\sum_{i=1}^{\mu(n)} c_{i,n}f_{i,n} \in C^\infty(M)$. 
Then by the estimates in the Lemma \ref{8.1} and \eqref{eqn;cnorm},
\begin{align}\label{15}
&|f_{\cc}|_{L^\infty(M)} \leq C|\textbf{c}|_{\ell^1}.\\
&\norm{f_{\cc}}_{\alpha,s} \leq Cvol{(U(t_{Ret}))}^{-1}(1+ \frac{\Sigma(t_{Ret})}{vol{(U(t_{Ret}))}}\norm{Y})^s|\textbf{c}|_s.
\end{align}
By orthogonality,
$$\norm{\left\langle \PPP^{d,\alpha}_{U(\bm{T})}\circ \QQ_y^{g,Y}, \omega_{\cc} \right\rangle}_{L^2(\T^g,dy)} = \left(\frac{vol{(U(\bm{T}))}}{vol{(U(t_{Ret}))}} \right)^{1/2} |\textbf{c}|_0.$$
By the estimate in Lemma \ref{8.1} for each $f_{i,n}$ , for every $z \in \T$ and all $\textbf{T} \in \R_+^d$, we have
$$\norm{\left\langle \PPP^{d,\alpha}_{U(\bm{T})}(\Phi_{\alpha,x}(\xi_{y,z})), \omega_{\cc} \right\rangle - \left\langle \PPP^{d,\alpha}_{U(\bm{T})}(\xi_{y,z}), \omega_{\cc} \right\rangle}_{L^2(\T^g,dy)} \leq 2|f_{\cc}|_{L^\infty(M)}\norm{Y}. $$
Let  $U(\textbf{T}_{Ret}) = [0,\textbf{T}_{Ret,1}]\times \cdots \times [0,\textbf{T}_{Ret,g}]$ where $\textbf{T}_{Ret,i} := t_{Ret,i}([T^{(i)}/t_{Ret,i}]+1)$. Then,
$$\norm{\left\langle \PPP^{d,\alpha}_{U(\bm{T})}(\xi_{y,z}), \omega_{\cc} \right\rangle - \left\langle \PPP^{d,\alpha}_{U(\textbf{T}_{Ret})}(\xi_{y,z}), \omega_{\cc} \right\rangle}_{L^2(\T^g,dy)} \leq vol{(U(t_{Ret}))}|f_{\cc}|_{L^\infty(M)}. $$
Therefore, there exists a constant $C'>0$ such that
\begin{multline*}
 \left| \norm{\left\langle \PPP^{d,\alpha}_{U(\bm{T})}(\Phi_{\alpha,x}(\xi_{y,z})), \omega_{\cc} \right\rangle}_{L^2(\T^g,dy)}  - \left(\frac{vol{(U(\bm{T}))}}{vol{(U(t_{Ret}))}} \right)^{1/2} |\textbf{c}|_0\right| \\
 \leq C' vol{(U(t_{Ret}))}|\cc|_{\ell^1}.
 \end{multline*}

For all $s >  s_{d,g} + 1/2$, by asymptotic property of Theorem \ref{6.2type}, for some constant $C_s>0$,
$$\Big|\left\langle \PPP^{d,\alpha}_{U(\bm{T})}m , \omega \right\rangle - \beta_H(\alpha,m,\bm{T})\D_\alpha^H(f_H) \Big | \leq C_s(1+L)\norm{f}_{\alpha,s}.$$
Combining previous estimates for $\beta_{\cc} = \beta^{f_{\cc}}$, we have
\begin{align*}
& \left|\norm{\beta_{\cc}(\alpha,\Phi_{\alpha,x}(\xi_{y,z}),\bm{T})}_{L^2(\T^g,dy)} - \left(\frac{vol{(U(\bm{T}))}}{vol{(U(t_{Ret}))}} \right)^{1/2} |\textbf{c}|_0\right| \\
& \leq C' vol{(U(t_{Ret}))}|\textbf{c}|_{\ell^1} + C_svol{(U(t_{Ret}))}^{-1}(1+L) \norm{f_{\cc}}_{\alpha,s}\\
& \leq  C_s'(vol{(U(t_{Ret}))} + vol{(U(t_{Ret}))}^{-1})(1+L)\big(1+ \frac{\Sigma(t_{Ret})}{vol{(U(t_{Ret}))}}\norm{Y}\big)^s|\textbf{c}|_{s}.
\end{align*}
Therefore, we derive the estimates in the statement.
\end{proof}

\section{Analyticity of functionals}\label{sec;6} %
In this section we prove that for all $\alpha \in DC$, the Bufetov functionals on any square are real analytic. 

\subsection{Analyticity}
By the orthogonal property of cocycle $\beta_H$ on an irreducible component $H = H_n$ with a central parameter $n \in \Z \backslash \{0\}$, the following statement is immediate. For any $(m,\textbf{T}) \in M \times \R_+^d$ and $t\in \R$,
\begin{equation}\label{orth}
\beta_H(\alpha,\phi_t^Z(m),\textbf{T}) = e^{2\pi \iota Knt  }\beta_H(\alpha,m,\textbf{T}).
\end{equation}

\begin{definition}\label{def;sstretch}
For every $t \in \R$, for each $1 \leq i \leq d$, and $m \in M$, the \emph{stretched (in direction of Z) rectangle} is denoted by
\begin{equation}
[\Gamma_{\bm{T}}]^Z_{i,t}(m) := \{(\phi^Z_{t s_i })\circ\PP^{d,\alpha}_{\bm{s}}(m) \mid \bm{s} \in U(\textbf{T})\}.
\end{equation}
\end{definition}

For $\s = (s_1, \cdots, s_d)\in \R^d$, let us denote by the standard rectangle
$\Gamma_\textbf{T}({\s}) :=  (\gamma_1(s_1), \cdots, \gamma_d(s_d))$ for $\gamma_i(s_i) = \exp(s_iX_i) $.
Similarly, we also write the stretched rectangle
\begin{equation}\label{def;stretch}
[\Gamma_{\bm{T}}]^Z_{i,t}({\s}) := (\gamma_1(s_1), \cdots,\gamma_{i,t}^Z(s_i), \cdots \gamma_d(s_d))
\end{equation}
where $\gamma_{i,t}^Z(s_i):= \phi^Z_{ts_i}(\gamma_i(s_i))$ is a stretched curve.

\begin{definition}
The \emph{restricted rectangle}  $\Gamma_{T,i,s}$ of the standard rectangle $\Gamma_\textbf{T}$ is defined as a restriction on $i$-th coordinate given by
$$\Gamma_{T,i,s}(\s) := \Gamma_\textbf{T}|_{U_{T,i,s}}(\s), $$
where  $U_{T,i,s} = [0,T^{(1)}] \times \cdots \times \underbrace{[0,s]}_{\text{i-th}} \cdots \times [0,T^{(d)}]$ for some $0<s \leq T^{(i)}$.
\end{definition}

\begin{lemma}\label{parts} For fixed elements $(X_i,Y_i,Z)$ satisfying commutation relation (\ref{commute}), the following formula for rank 1 action holds:
$$\hat\beta_H(\alpha,[\Gamma_{\bm{T}}]^Z_{i,t}) = e^{2\pi \iota t n KT^{(i)}}\hat\beta_H(\alpha,\Gamma_{\bm{T}}) - 2\pi \iota n K t \int_0^{T^{(i)}}e^{2\pi \iota n K t s_i}\hat\beta_H(\alpha,\Gamma_{T,i,s})ds_i.$$
\end{lemma}
\begin{proof}
 Let $\alpha = (X_i, Y_i,Z)$ and $\omega$ be $d$-form supported on a single irreducible representation $H$.
We obtain the following formula for stretches of curve $\gamma_{i,t}^Z$ (see \cite[\S 4 and Lemma 9.1]{forni2020time}),
$$\frac{d\gamma_{i,t}^Z}{ds_i} = D\phi^Z_{ts_i}(\frac{d\gamma_i}{ds_i}) + tZ\circ \gamma_{i,t}^Z. $$

It follows that pairing is given by
\begin{align*}
\langle[\Gamma_\textbf{T}]^Z_{i,t}, \omega \rangle & = \int_{U(\bm{T})}\omega \Big(\frac{d\gamma_1}{ds_1}(s_1), \cdots, \frac{d\gamma_{i,t}^Z}{ds_i}(s_i) ,\cdots, \frac{d\gamma_d}{ds_d}(s_d)\Big) d\s \\
& = \int_{U(\bm{T})}e^{2\pi \iota n Kts_i}[\omega\Big(\frac{d\gamma_1}{ds_1}(s_1), \cdots, \frac{d\gamma_d}{ds_d}(s_d)\Big)] + \iota_Z\omega \circ [\Gamma_\textbf{T}]^Z_{i,t}(\s)d\s
\end{align*}
Denote a $(d-1)$-dimensional sub-rectangle $U_{d-1}(\bm{T})$ without $i$-th coordinates by $U(\bm{T}) = U_{d-1}(\bm{T}) \times [0,T^{(i)}]$. Integration by parts for a fixed $i$-th integral gives
\begin{align*}
&\int_{U(\bm{T})}e^{2\pi \iota n Kts_i}[\omega(\frac{d\gamma_1}{ds_1}(s_1), \cdots, \frac{d\gamma_d}{ds_d}(s_d))]d\s  \\
&= e^{2\pi \iota  n Kt{T^{(i)}}}\int_{U(\bm{T})}[\omega(\frac{d\gamma_1}{ds_1}(s_1), \cdots, \frac{d\gamma_d}{ds_d}(s_d))]d\s\\
& \quad- 2\pi \iota n Kt \int_{0}^{T^{(i)}}  e^{2\pi \iota  n Kts_i}\\&\quad \int_{U_{d-1}(\bm{T})}\left(\int_{0}^{s_i} [\omega(\frac{d\gamma_1}{ds_1}(s_1), \cdots,\frac{d\gamma_i}{ds_i}(r) \cdots, \frac{d\gamma_d}{ds_d}(s_d))]dr\right)d\s.
\end{align*}
Then, we have the following formula
\begin{multline*}
 \langle [\Gamma_\textbf{T}]_{i,t}^Z, \omega \rangle = e^{2\pi \iota  n Kt{T^{(i)}}}\langle [\Gamma_\textbf{T}], \omega \rangle - 2\pi \iota n Kt \int_{0}^{T^{(i)}} e^{2\pi \iota  n Kts_i} \langle\Gamma_{T,i,s}, \omega   \rangle ds_i\\
 +\int_{U(\bm{T})}(\iota_Z\omega \circ [\Gamma_\textbf{T}]^Z_{i,t})(\s)d\s.
 \end{multline*}
Since the action of $\PP_\textbf{t}^{d,X}$ for $\textbf{t} \in \R^d$ is identity on the center $Z$,
$$\lim_{t_1, \dotsc, t_d \rightarrow \infty} e^{-(t_1+ \cdots t_d)/2}\int_{U(\bm{T})} (\iota_Z(\PP_\textbf{t}^{d,X})^*\omega \circ [\Gamma_\textbf{T}]^Z_{i,t})(\s) d\s = 0.$$
Thus, in particular concerning $d=1$, it follows by definition of the functional (Lemma \ref{35}), the statement holds.
\end{proof}

Here we define a restricted vector $\textbf{T}_{i,s}$ of $\textbf{T} = (T^{(1)},\cdots, T^{(d)}) \in \R^d$. For fixed $i \in [1,d]$, pick $s_i \in [0,T^{(i)}]$ such that
$\textbf{T}_{i,s} \in \R^d$ is a vector with its coordinates
$$
{T}^{(j)}_{i,s} =
\begin{cases}
  {T}^{(j)} & \text{if} \quad j \neq i\\
  s_i & \text{if} \quad  j = i.
\end{cases}
$$
Similarly, $\textbf{T}_{i_1,\cdots, i_k ,s}$ is a vector whose $i_1, \cdots, i_k$-th coordinates are replaced by $s_{i_1}, \cdots, s_{i_k}$.

\begin{lemma}\label{9.2} For $m\in M$ and $y_i \in \R$, the following property holds:
\begin{multline*}
\beta_H(\alpha,\phi_{y_i}^{Y_i}(m),\textbf{T}) = \\
e^{-2\pi \iota y_i n K{T}^{(i)}} \beta_H(\alpha,m,\textbf{T}) + 2\pi \iota  n Ky_i \int_0^{{T}^{(i)}} e^{-2\pi \iota y_i n Ks_i} \beta_H(\alpha,m,\textbf{T}_{i,s})ds_i.
\end{multline*}
\end{lemma}
\begin{proof}
By definition \eqref{eqn;rect}, \eqref{def;stretch} and
commutation relation (\ref{commute}), it follows that
$$\phi_{y_i}^{Y_i}(\Gamma_\textbf{T}^X(m)) = [\Gamma_\textbf{T}^X(\phi_{y_i}^{Y_i}(m))]_{i,t}^Z.$$
By the invariance property of the functional $\hat\beta_H$ and Lemma \ref{parts},
\begin{align*}
& \beta_H(\alpha,m,\textbf{T})  = \hat\beta_H(\alpha,\phi_{y_i}^{Y_i}(\Gamma_\textbf{T}^X(m))  \\
& = e^{2\pi \iota y_i n K{T}^{(i)}}\hat\beta_H(\alpha,\Gamma_\textbf{T}^X(\phi_{y_i}^{Y_i}(m)) \\&\quad - 2\pi \iota n Ky_i \int_0^{{T}^{(i)}}  e^{2\pi \iota  n Ky_is_i}\hat\beta_H(\alpha,\Gamma_{T,i,s}^X(\phi_{s_i}^{Y_i}(m))ds_i\\
& = e^{2\pi \iota y_i n K{T}^{(i)}}\beta_H(\alpha,\phi_{y_i}^{Y_i}(m),\textbf{T}) - 2\pi \iota n Ky_i \int_0^{{T}^{(i)}}  e^{2\pi \iota  n Ky_is_i}\beta_H(\alpha,\phi_{s_i}^{Y_i}(m),\textbf{T}_{i,s})ds_i.
\end{align*}
Then the statement follows immediately.
\end{proof}

We extend previous Lemma \ref{9.2} to higher rank actions by induction argument.
\begin{lemma}[Rank-$d$ action]\label{9.3}
For $m\in M$ and $y = (y_1, \cdots, y_d) \in \R^d$, the following identity for the cocycle $\beta_H$ holds:
\begin{align}\label{eqn;lem9.3}
\begin{split}
&\beta_H(\alpha,\QQ_y^{d,Y}(m) ,\textbf{T})  = e^{-2\pi \iota \sum_{j=1}^d y_j n K{T}^{(j)}} \beta_H(\alpha,m,\textbf{T}) \\
&+ \sum_{k=1}^d\sum_{\substack{1 \leq i_1 < \cdots < i_k \leq d} }\prod_{j=1}^{k}(2\pi \iota n Ky_{i_j})e^{-2\pi \iota  n K(\sum_{\substack{l\notin \{i_1, \cdots, i_k\}}}y_{l}{T}^{(l)}) } \\
& \times   \int_0^{{T}^{(i_1)}}\cdots\int_0^{{T}^{(i_k)}}e^{-2\pi \iota  n K(y_{i_1}s_{i_1}+\cdots + y_{i_k}s_{i_k} )}\beta_H(\alpha,m,\textbf{T}_{i_1,\cdots, i_k ,s})ds_{i_k}\cdots ds_{i_1}.
\end{split}
\end{align}
\end{lemma}

\begin{proof}
We verified that the statement works for $d=1$ in Lemma \ref{9.2}.
Assume that \eqref{eqn;lem9.3} holds for rank $d-1$ action $\QQ_{y'}^{d-1,Y}$ by induction hypothesis. For convenience, we write
$$\QQ_y^{d,Y}(m) = \phi_{y_d}^{Y_d}\circ\QQ_{y'}^{d-1,Y}(m) \text{ for } y' \in \R^{d-1} \text{ and } y = (y',y_d) \in \R^d.$$
By applying Lemma \ref{9.2},
\begin{align}\label{eqn;9.2}
\begin{split}
\beta_H(\alpha,\QQ_y^{d,Y}(m),\textbf{T}) & = e^{-2\pi \iota y_d n K{T}^{(d)}} \beta_H(\alpha, \QQ_{y'}^{d-1,Y}(m),\textbf{T})\\
& + 2\pi \iota  n Ky_d \int_0^{{T}^{(d)}} e^{-2\pi \iota y_d n Ks_d} \beta_H(\alpha,\QQ_{y'}^{d-1,Y}(m),\textbf{T}_{d,s})ds_d  \\
& := I + II.
\end{split}
\end{align}
Firstly, by induction hypothesis,
\begin{align*}
I &=  e^{-2\pi \iota y_d n K{T}^{(d)}}\Bigg(e^{-2\pi \iota \sum_{j=1}^{d-1} y_j n K{T}^{(j)}} \beta_H(\alpha,m,\textbf{T}) \\
&+ \sum_{k=1}^{d-1}\sum_{\substack{1 \leq i_1 < \cdots < i_k \leq d-1} }\prod_{j=1}^{k}(2\pi \iota n Ky_{i_j})e^{-2\pi \iota  n K(\sum_{\substack{l\notin \{i_1, \cdots, i_k\}}}y_{l}{T}^{(l)} + y_d{T}^{(d)} )} \\
& \times   \int_0^{{T}^{(i_1)}}\cdots\int_0^{{T}^{(i_k)}}e^{-2\pi \iota  n K(y_{i_1}s_{i_1}+\cdots + y_{i_k}s_{i_k} )}\beta_H(\alpha,m,\textbf{T}_{i_1,\cdots, i_k ,s})ds_{i_k}\cdots ds_{i_1}\Bigg)\\
& = e^{-2\pi \iota \sum_{j=1}^d y_j n K{T}^{(j)}} \beta_H(\alpha,m,\textbf{T}) + III.
\end{align*}
Then term $III$ contains iterated integrals on the restricted rectangles (from 0 to $(d-1)$-th) containing a term $e^{-2\pi \iota  n K y_d{T}^{(d)} }$ outside of iterated integrals.

 For the second part,  we apply induction hypothesis again for restricted rectangle $\textbf{T}_{d,s}$. Then,
\begin{align*}
II &=  2\pi \iota  n Ky_d \int_0^{{T}^{(d)}} e^{-2\pi \iota y_d n Ks_d} \big[e^{-2\pi \iota \sum_{j=1}^{d-1} y_j n K{T}^{(j)}} \beta_H(\alpha,m,\textbf{T}_{d,s})]ds_{d}\\
&+ \sum_{k=1}^{d-1}\sum_{\substack{1 \leq i_1 < \cdots < i_k \leq d-1} } (2\pi \iota  n Ky_d) \prod_{j=1}^{k}(2\pi \iota n Ky_{i_j})e^{-2\pi \iota  n K(\sum_{\substack{l\notin \{i_1, \cdots, i_k\}}}y_{l}{T}^{(l)} )} \\
& \times   \int_0^{{T}^{(d)}} \left( \int_0^{{T}^{(i_1)}}\cdots\int_0^{{T}^{(i_k)}} ds_{i_k}\cdots ds_{i_1} \right)ds_{d}  \\
& \times e^{-2\pi \iota  n K(y_{i_1}s_{i_1}+\cdots + y_{i_k}s_{i_k}+y_ds_d )}\beta_H(\alpha,m,\textbf{T}_{i_1,\cdots, i_k,d ,s}).
\end{align*}
The term $II$ consist of 1 to $d$-th iterated integrals on the (from 1 to $d$-th) restricted rectangles containing $e^{-2\pi \iota  n K y_ds_d }$ inside of iterated integrals. Thus, by rearranging terms $II$ and $III$, we obtain all the terms in the expression \eqref{eqn;lem9.3}.
\end{proof}

\subsection{Extensions of domain}
In this subsection, the domain of functionals defined on standard rectangle $\Gamma^X_\textbf{T}$ extends to the class $\mathfrak{R}$ (see Definition \ref{def;rect}). Furthermore, the functional associated with the analytic norm extends to holomorphic function on a complex domain.
\begin{proof}[{Proof of Corollary \ref{cor;extend}.}]
Firstly, we can extend our functional to class $(\QQ_y^{d,Y})_*\Gamma^X_\textbf{T}$ for any $y \in \R^d$ by invariance property (Lemma \ref{lem;invar}).
Similarly, by Lemma \ref{parts}, Bufetov functional defined on the standard rectangles extends to the class of generalized rectangle $(\phi^Z_{t_i z})\circ\PP^{d,\alpha}_{\bm{t}}(m)$.
Since the flow generated by $Z$ commutes with other actions $\PP$ and $\QQ$, for any standard rectangle $\Gamma = \Gamma(m)$ with a fixed point $m \in M$, we have $(\phi^Z_z)_*\Gamma(m) = \Gamma(\phi^Z_z(m))$ for any $z \in \R$. Therefore, by combining with the invariance under the action $\QQ$ from Lemma \ref{lem;invar}, the domain of Bufetov functional extends to the class $\mathfrak{R}$.
\end{proof}

For any $R>0$, the \emph{analytic norm} is defined for all ${\bf c} \in \ell^2$ as
$$\norm{\bf{c}}_{\omega, R} = \sum_{n\neq 0}\sum_{i=1}^{\mu(n)}e^{nR}|c_{i,n}|.$$

 Let $\Omega_R$ denote the subspace of $\textbf{c} \in \ell^2$ such that $\norm{\textbf{c}}_{\omega, R}$ is finite.
\begin{lemma}\label{16} For $\textbf{c} \in \Omega_R$ and $\textbf{T}\in \R_{+}^d$, the function 
$$\beta_\textbf{c}(\alpha,\QQ_y^{d,Y}\circ \phi_z^Z(m) ,\textbf{T}), \ (y,z) \in \R^d\times \T$$
extends to a holomorphic function in the domain 
\begin{equation}\label{eqn;D_{R,T}}
D_{R,T} := \{(y,z) \in \C^d\times \C/\Z \mid \sum_{i=1}^d |Im(y_i)|{T}^{(i)} + |Im(z)|< \frac{R}{2\pi K}  \}.
\end{equation}

The following bound holds: for any $R'<R$ there exists a constant $C>0$ such that, for all $(y,z) \in D_{R',T}$ we have
  \begin{multline*}
|\beta_\textbf{c}(\alpha,\QQ_y^{d,Y}\circ \phi_z^Z(m),\textbf{T})| \\
\leq  C_{R,R'}\norm{\bf{c}}_{\omega, R}(L + vol(U(\bm{T}))^{1/2}(1+E_M(a,\bm{T})) (1+K\sum_{i=1}^d |Im(y_i)|{T}^{(i)}).
  \end{multline*}
\end{lemma}
\begin{proof}

By Lemma \ref{9.3} and (\ref{orth}),
\begin{align*}
&\beta_\textbf{c}(\alpha,\QQ_y^{d,Y}\circ \phi_z^Z(m) ,\textbf{T})  = e^{(z-2\pi \iota \sum_{j=1}^d y_j n K{T}^{(j)})} \beta_H(\alpha,m,\textbf{T}) \\
&+ \sum_{k=1}^d\sum_{\substack{1 \leq i_1 < \cdots < i_k \leq d} }\prod_{j=1}^{k}(2\pi \iota n Ky_{i_j})e^{-2\pi \iota  n K(\sum_{\substack{l\notin \{i_1, \cdots, i_k\}}}y_{l}{T}^{(l)} )} \\
& \times    e^{2\pi \iota  n Kz}\int_0^{{T}^{(i_1)}}\cdots\\&\int_0^{{T}^{(i_k)}}e^{-2\pi \iota  n K(y_{i_1}s_{i_1}+\cdots + y_{i_k}s_{i_k} )}\beta_H(\alpha,m,\textbf{T}_{i_1,\cdots, i_k ,s})ds_{i_k}\cdots ds_{i_1}.
\end{align*}

In view of Lemma \ref{6.1type}, for each variable $(y_i,z) \in \C\times \C/\Z$, we have 
\begin{align*}
&|\beta_\textbf{c}(\alpha,\QQ_y^{d,Y}\circ \phi_z^Z(x) ,\textbf{T})| \\
& \leq  (L + vol(U(\bm{T}))^{1/2}(1+E_M(a,\bm{T}))\Bigg(C_1\sum_{n\neq 0}\sum_{i=1}^{\mu(n)}e^{nR}|c_{i,n}|e^{2\pi |Im(z - \sum_{i=1}^d{T}^{(i)}y_i) |nK} \\
& + \sum_{k=1}^d C_k \Big(\sum_{1 \leq i_1 < \cdots < i_k \leq d }\prod_{j=1}^{k}(|Im(y_{i_j})|{T}^{(i_j)}) \\
& \times \sum_{n\neq 0}\sum_{i=1}^{\mu(n)}n|c_{i,n}|e^{2\pi (|Im(z)| + \sum_{j=1}^k {T}^{(i_j)}|Im(y_{i_j})|)nK}\Big)\Bigg).
\end{align*}
Therefore, the functional $\beta_\textbf{c}(\alpha,\QQ_y^{d,Y}\circ \phi_z^Z(m) ,\textbf{T})$ is bounded by a series of holomorphic functions on $\C^d \times \C / \Z$ and it converges uniformly on compact subsets of domain $D_{R,T}$. Thus it is holomorphic on the set $D_{R,T}$.
\end{proof}

\section{Measure estimation for bounded-type}\label{sec;7}
In this section, we prove a measure estimation of Bufetov functional with an automorphism $\alpha$ of bounded-type. This result is a generalization of \S 11 in \cite{forni2020time}.

Let $\OO_r$ denote the space of holomorphic functions on the ball $B_\C(0,r) \subset \C^n$.
\begin{theorem}\cite[Theorem 1.9]{bru}\label{ball}   For any $f \in \OO_r$, there is a constant $d:= d_f(r)>0$ such that for any convex set $D \subset B_\R(0,1) := B_\C(0,1) \cap \R^n$, for any measurable subset $U \subset D$
$$\sup_{D}|f| \leq \left(\frac{4n Leb(D)}{Leb(U)} \right)^d\sup_{U}|f|.$$
\end{theorem}
We say that a holomorphic function $f$ defined in a disk is \emph{$p$-valent}
if it assumes no value more than $p$-times there. We also say that $f$ is \emph{0-valent} if it
is a constant.

\begin{definition}\cite[Definition 1.6]{bru}  Let $\LL_t$ denote the set of one-dimensional complex affine spaces $L \subset \C^n$ such that $L \cap  B_\C(0,t) \neq \emptyset$.
For $f \in \OO_r$, the number
$$\nu_f(t):=\sup_{L \in \LL_t}\{\text{valency of }f \mid L \cap  B_\C(0,t) \neq \emptyset \}$$
is called the \emph{valency} of $f$ in $B_\C(0,t)$.
\end{definition}
By Proposition 1.7 of \cite{bru},  for any $f \in \OO_r$ with finite valency $\nu_f(t)$ for any $t \in [1,r)$, there is a constant $c := c(r) >0$ such that
\begin{equation}\label{chev}
d_f(r) \leq c\nu_f(\frac{1+r}{2}).
\end{equation}

\begin{lemma}\cite[Lemma 10.3]{forni2020time}\label{lem;lem10.3} Let $R>r>1$. For any normal family $\F \subset \OO_R$, assume that no functions in $\overline \F = \emptyset$ are constant along a one-dimensional complex line. Then we have
\[ \sup_{f \in \F} \nu_f(r) < \infty.\]
\end{lemma}

\begin{lemma}\label{11.1}
Let $L>0$ and $\B \subset DC(L)$ be a bounded subset. Given $R>0$, for all $\textbf{c} \in \Omega_R$ and all $\textbf{T}^{(i)}>0$, denote $\F(\textbf{c},\bm{T})$ by the family of real analytic functions of the variable $y \in [0,1)^d$ and
$$\F(\textbf{c},\bm{T}):= \{ \beta_{\textbf{c}}(\alpha,\Phi_{\alpha,x}(\xi_{y,z}), \bm{T}) \mid (\alpha,x,z) \in \B \times M  \times \T \}. $$
Then there exist $\textbf{T}_\B := (\textbf{T}^{(i)}_\B)$ and $\rho_\B >0$, such that for every $(R,\bm{T})$ with  $R/\textbf{T}^{(i)} \geq \rho_\B$, $\textbf{T}^{(i)} \geq \textbf{T}^{(i)}_\B$ and for all $\textbf{c} \in \Omega_R\backslash \{0\}$, we have
\begin{equation}\label{eqn;finiteval}
\sup_{f \in \F(\textbf{c},\bm{T})} \nu_f < \infty.
\end{equation}
\end{lemma}

\begin{proof}
Since $\B \subset \M$ is bounded, for each time $t_i \in \R$ and $1 \leq i \leq g$,
$$0 < t_{\B}^{min} = \min_i\inf_{\alpha \in \B} t_{Ret,i,\alpha} \leq \max_i\sup_{\alpha \in \B} t_{Ret,i,\alpha} = t_{\B}^{max} < \infty.$$

For any $\alpha \in \B$ and $x \in M$, the map $\Phi_{\alpha,x} : [0,1)^d \times \T \rightarrow \prod_{i=1}^d[0,t_{\alpha,i}) \times \T$ in (\ref{11}) extends to a complex analytic diffeomorphism $\hat\Phi_{\alpha,x}: \C^d \times \C / \Z \rightarrow \C^d \times \C / \Z$. By Lemma \ref{16}, it follows that for fixed $z \in \T$, real analytic function $\beta_{\cc}(\alpha,\Phi_{\alpha,x}(\xi_{y,z}), \bm{T})$
extends to a holomorphic function defined on a region
$$H_{\alpha,m,R,t} := \{y \in \C^d \mid \sum_{i=1}^d |Im(y_i)| \leq h_{\alpha,m,R,t} \}. $$
By boundedness of the set $\B \subset \M$, it follows that
$$\inf_{(\alpha,x) \in \B \times M} h_{\alpha,m,R,t}:= h_{R,T} >0.$$
We remark that the function $h_{\alpha,m,R,t}$ and its lower bound $h_{R,T}$ can be obtained from the formula \eqref{11} for the polynomial $\Phi_{\alpha,x}$ and the definition of the domain $D_{R,T}$ in the formula \eqref{eqn;D_{R,T}}.

For every $r>1$, there exists $\rho_\B >1$ such that for every $(R,\textbf{T})$ with $R/\textbf{T}^{(i)} > \rho_\B$,
\begin{equation}\label{eqn;oo}
\beta_{\cc}(\alpha,\Phi_{\alpha,x}(\xi_{y,z}), \textbf{T}) \in \OO_r
\end{equation}
as a function of $y \in \T^d$.

By Lemma \ref{16}, the family $\F(c,\textbf{T})$ is uniformly bounded and normal. By Lemma \ref{Crucial} for the non-zero $L^2$-lower bound of functionals, for sufficiently large pair \textbf{T}, no sequence from $\F(c,\textbf{T})$ can converge to a constant. Therefore, by Lemma  \ref{lem;lem10.3} for the family $\F = \F(c,\textbf{T})$, the main statement follows.
\end{proof}

We derive measure estimates of Bufetov functionals on the rectangular domain. 

\begin{lemma}\label{lem:prev}Let $\alpha \in DC$ such that the forward orbit of $\R^d$-action $\{r_{\bm{t}}[\alpha]\}_{\bm{t} \in \R^d_+}$ is contained in a compact set of $\M_g$. There exist $R, C, \delta>0$ and $\bm{T}_0 \in \R_+^d $ such that, for every $\textbf{c} \in \Omega_R \backslash \{0\}$, $\bm{T} \geq \bm{T}_0$ and for every $\epsilon >0$, we have
$$vol(\{m \in M \mid  |\beta_\textbf{c}(\alpha, m ,\bm{T})| \leq \epsilon vol(U(\bm{T}))^{1/2}  \}) \leq C\epsilon^{\delta}. $$
\end{lemma}
\proof
Since $\alpha \in DC$ and the orbit $\{r_{\bm{t}}[\alpha]\}_{\bm{t} \in \R_+}$ is contained in a compact set, there exists $L>0$ such that $r_{\bf{t}}[\alpha] \in DC(L)$ for all $\bm{t} \in \R^d_+$. Then, we choose $\textbf{T}_0 \in \R^d$ and $R>0$ from the conclusion of Lemma \ref{11.1}. By the scaling property,
$$\beta_\textbf{c}(\alpha,m ,\bm{T}) = \left(\frac{vol(U(\bm{T}))}{vol(U(\bm{T_0}))} \right)^{1/2}\beta_\textbf{c}(g_{\log(\bm{T/T_0})}[\alpha],m , \bm{T}_0).$$
By Fubini's theorem, it suffices to estimate
$$Leb(\{y \in [0,1]^d \mid  |\beta_\textbf{c}(\alpha,\Phi_{\alpha,x}(\xi_{y,z}), \bm{T}_0)| \leq \epsilon  \}). $$
Let $\delta^{-1} := c(r)\sup_{f \in \F(\textbf{c},\textbf{T}_0)} \nu_f(\frac{1+r}{2})< \infty $ as in \eqref{chev} and \eqref{eqn;finiteval}. By Lemma \ref{Crucial}, we have
$$\inf_{(\alpha,x,z) \in \B \times M \times \T} \sup_{y \in [0,1]^d}|\beta_\textbf{c}(\alpha,\Phi_{\alpha,x}(\xi_{y,z}) ,\bm{T}_0)|>0 $$
so that the functional is not trivial.
By Theorem \ref{ball} for the unit ball $D = B_\R(0,1)$ and setting
$$U = \{y \in [0,1]^d \mid  |\beta_\textbf{c}(\alpha,\Phi_{\alpha,x}(\xi_{y,z}), \bm{T}_0)| \leq \epsilon  \},$$
by the bound in \eqref{chev} for $d_f(r)$, there exists a constant $C>0$ such that for all $\epsilon >0$ and $(\alpha,x,z) \in \B \times M \times \T$,
$$Leb(\{y \in [0,1]^d \mid  |\beta_\textbf{c}(\alpha,\Phi_{\alpha,x}(\xi_{y,z}) , \bm{T}_0)| \leq \epsilon  \}) \leq C\epsilon^{\delta}. $$
Then the statement follows from the Fubini theorem.
\qed

\begin{corollary}\label{cor;volume} Let $\alpha$ be as in the previous Lemma \ref{lem:prev}. There exist $R, C, \delta>0$ and $\bm{T}_0 \in \R_+^d $ such that, for every $\textbf{c} \in \Omega_R \backslash \{0\}$, $\bm{T} \geq \bm{T}_0$ and for every $\epsilon >0$, we have
$$vol\left(\{m\in M \mid |\langle \PPP^{d,\alpha}_{U(\bm{T})}m , \omega_{\cc} \rangle | \leq \epsilon vol{(U(\bm{T}))}^{1/2}  \}\right) \leq C\epsilon^{\delta}.$$
\end{corollary}

\section*{Acknowledgments} The author deeply appreciates Giovanni Forni for his valuable suggestions to improve the draft. He acknowledges Daniel Sell, Rodrigo Trevi\~no, and Corinna Ulcigrai for giving several comments. He is also grateful to Oliver Butterley, Jacky Jia Chong, Krzysztof Fr\c aczek, Osama Khail, and Lucia D. Simonelli for fruitful discussions. This work was initiated when the author visited the Institut de Mathematiques de Jussieu-Paris Rive Gauche in Paris, France. He acknowledges invitation and hospitality during the visit. Lastly, the author is thankful to the referee for helpful comments and suggestions for improvement in the presentation of this work.

\bibliographystyle{amsalpha}

\end{document}